\documentclass[11pt]{amsart}
\usepackage{amscd, amssymb, amsmath}
\input xy
\xyoption{all}

\newtheorem{Thm}{Theorem}[subsection]

\newtheorem{Def/Thm}[Thm]{Definition/Theorem}
\newtheorem{Cor}[Thm]{Corollary}
\newtheorem{Lemma}[Thm]{Lemma}

\theoremstyle{remark}

\numberwithin{equation}{subsection}

\newcommand{\ot }{\otimes}
\newcommand{\ra }{\rightarrow}

\newcommand{\Spec}{{\mathrm{Spec}}}

\newcommand{\Sch}{{\mathrm{Sch}}}

\newcommand{\fX}{{\mathfrak{X}}}

\newcommand{\fM}{{\mathfrak{M}}}
\newcommand{\fMB}{{\mathfrak{MB}}}
\newcommand{\sm}{{\mathrm{sm}}}
\newcommand{\sing}{{\mathrm{sing}}}
\newcommand{\id}{{\mathrm{id}}}

\newcommand{\LMmu}{\overline{M}_{g,n,\mu }^{\mathrm{log}}(\fX^ +/\fX ,\beta )}
\newcommand{\co}{{\mathcal{O}}}

\newcommand{\cF}{{\mathcal{F}}}
\newcommand{\cL}{{\mathcal{L}}}

\newcommand{\cB}{{\mathcal{B}}}
\newcommand{\cM}{{\mathcal{M}}}
\newcommand{\NN}{{\mathbb N}}

\newcommand{\PP }{{\mathbb P}}

\newcommand{\CC }{{\mathbb C}}
\newcommand{\ZZ }{{\mathbb Z}}

\begin{document}

\title{Logarithmic Stable Maps}

\author{Bumsig Kim}
\address{School of Mathematics, Korea Institute for Advanced Study,
Hoegiro 87, Dongdaemun-gu, Seoul, 130-722, Korea}
\email{bumsig@kias.re.kr}

\begin{abstract}
We introduce the notion of a logarithmic stable map from a
minimal log prestable curve to a log twisted semi-stable
variety of form $xy=0$. We study the compactification of the moduli spaces of such maps
and provide a perfect obstruction theory, applicable to the moduli spaces of (un)ramified stable
maps and  stable relative maps. As an application, we obtain a modular
desingularization of the main component of Kontsevich's moduli
space of elliptic stable maps to a projective space.
\end{abstract}


\maketitle


\section{Introduction}
\subsection{}
In papers \cite{HM, Li1}, the admissibility, or equivalently, the
predeformability of J. Li was introduced. It is  a condition on
maps from curves to semi-stable varieties which are \'etale
locally of form $xy=0$. The condition is natural.
It is a necessary and sufficient condition
 to deform, \'etale locally at  the domain,
 such a map to a map from a smooth domain curve to a smooth target.
 It is, however, not a condition friendly to moduli problems
(\cite{Li1, Li2, GV}). If suitable log structures are given both
on the domain curve and the target space, the admissibility
amounts to the requirement that the map is a log morphism which is
simple at the inverse image of the singular locus. For the
separatedness of moduli spaces of such maps, it will be imposed
that the log structures on the domain prestable curves are minimal
(see \ref{MinimalCurve}); the log structures on the targets are
extended log twisted ones (see \ref{ExtendedLogTwistedSpace}); and
the automorphism groups of the log morphisms are finite. Those
maps, satisfying the above conditions, will be called log stable
maps. Once suitable log structures are introduced, due to the
construction of log cotangent complexes by Olsson in
\cite{Olsson_Cotangent} among other things, it is straightforward
to show that the moduli spaces of log stable maps carry the
relative perfect obstruction theory  identical to that of usual
stable maps case if the tangent sheaves of the targets are
replaced by the log tangent sheaves (see \ref{ob1}). In log sense,
both sources and targets of the maps we consider are smooth, which
is one of reasons why log geometry works well in the study of
moduli spaces of such stable maps. The advantage of the usage of
log geometry  is manifest in \ref{ob2}: The log admissibility
suffices to deform, \'etale locally at  the target,
 such a map to a map from a smooth domain curve to a smooth target.

We explain the results of this paper, more precisely. Let
$\mathbf{k}$ be a fixed algebraically closed field of
characteristic zero. Let $\cB$ be an algebraic stack whose objects
form a collection of $(W/S, W\ra X)$, where $S, X$ are schemes
over a base algebraic $\mathbf{k}$ scheme $\Lambda$; $W$ is an
algebraic space over $S$;
 $W/S$ is a proper flat family
 of semi-stable varieties of form $xy=0$;
 and $W\ra X$ is a
 map. Assume that $\cB$ is smooth over $\Lambda$
 and has a universal family $\mathcal{U}$ over $\cB$. Main examples we have in
 mind are these: the stack $\fX$ of FM spaces of a smooth projective variety $X$ (\cite{KKO});
  the stack of expanded degeneration spaces
 of a smooth projective variety with respect to a smooth divisor (\cite{Li1, KS});
 and the stack of expanded degeneration spaces of a semi-stable, projective, degeneration space of form $xy=0$
 (\cite{Li1}).

\medskip

\noindent{\bf Main Theorem A.} {\em The moduli stack  $\overline{M}_{g,n}^{\mathrm{log}}(\mathcal{U/B})$
 of $(g,n)$ log stable  maps to $W/S \in \cB$
  is an algebraic stack with a perfect obstruction theory over $\Lambda$.
  Furthermore if the moduli stack
   $\overline{M}_{g,n}(\mathcal{U/B})$
of its underlying stable maps is a proper DM stack over $\Lambda$, so is
$\overline{M}_{g,n}^{\mathrm{log}}(\mathcal{U/B})$ over $\Lambda$.}

\medskip

We refer to subsection  \ref{StackLogStableMaps} for the precise
meaning of the statement of Main Theorem A. The theorem yields an
explicit description of a perfect obstruction theory for the log versions of the stack
of stable (un)ramified maps (\cite{KKO}) and the stack of stable
relative maps of J. Li (\cite{Li2}). In fact, the moduli stack $\overline{M}_{g,n}^{\mathrm{log}}(\mathcal{U/B})$ is
equipped with a log structure, and hence it is a log algebraic
stack.
When $\cB$
is the stack $\fX$ of FM spaces of a projective nonsingular
variety $X$ over $\bf k$ (\cite{KKO}), after (un)ramified condition and the
strong stability being imposed, the stack
$\overline{M}_{g,n}^{\mathrm{log}}(\fX ^+/\fX)$ is a proper DM log
stack over $\mathbf{k}$. Here $\fX ^+$ is the universal family of $\fX$.

When the genus is 1, we will consider a  variant of stable (un)ramified maps,   namely,
 elliptic log stable maps to chain type FM spaces of $X$. The key condition on these,
possibly non-finite, maps is that either the genus 1 components or the loops of the rational components
are nonconstantly mapped under the maps. See \ref{chain1} for the precise definition.

\medskip

\noindent{\bf Main Theorem B.} {\em   The moduli stack
$\overline{M}_{1}^{\mathrm{log,\, ch}}(\fX ^+/\fX )$ of elliptic log stable maps to chain type FM spaces of $X$ is a
proper DM-stack over $\mathbf{k}$, carrying a perfect obstruction theory.
When $X$ is a projective space $\PP ^r_{\mathbf{k}}$,
the stack is smooth over $\mathbf{k}$.
 }
\medskip

Here the smoothness means the usual smoothness as a DM stack.
Hence it provides a moduli-theoretic desingularization of the main
component of stable maps. It would be interesting to find an
explicit relationship with Vakil and Zinger's desingularization in
\cite{VZ}. The space $\overline{M}_{1}^{\mathrm{log,\, ch}}(\fX
^+/\fX)$
 can be perhaps used to algebro-geometrically
establish quantum Lefschetz hyperplane section principle for
elliptic case and to prove elliptic Mirror Conjecture for
Calabi-Yau hypersurfaces
 in a projective space.
 The hyperplane section principle for reduced genus 1 Gromov-Witten
 invariants and
the comparisons between reduced and standard genus 1 invariants
are accomplished in \cite{LZ, Z1, Z2, Z3}.

\subsection{Remarks}
The minimality condition on log prestable curves, defined here, is
a condition stronger than the minimality introduced by Wewers in
\cite{We}. The idea that log structures should be useful in
(relative) Gromov-Witten theory has been around many years
(\cite{Li2, Si}). After finishing the paper, the author learned of
Siebert's long unfinished project  \cite{Si} on log GW invariants
that includes a construction of the virtual fundamental class
without using log cotangent complexes. After posting the paper,
the author was informed that there is another approach by
Abramovich and Fantechi using twisted stacky structures along the
singular loci \cite{AF}.

\subsection{Conventions}

 Throughout the paper, all schemes are locally
 noetherian schemes over $\Lambda$ unless otherwise  stated.
 The readers who are familiar with log geometry can skip section \ref{LogGeometry}, keeping
 in mind the following conventions.
 Log structures considered are always fine.
 Log schemes will be denoted by $(X,M)$ (or $X^\dagger$) when
 $X$ is an underlying scheme and $M$ is a sheaf of monoid
 on the \'etale site $X_{\text{\'et}}$ of $X$, equipped with a homomorphism
 $\alpha : M\ra \mathcal{O}_X$ such that $\alpha ^{-1}(\co _X^\times ) \cong \co _X^\times$
 under $\alpha$.
  Sometimes $X$ alone will denote the log scheme. For a monoid $P$,
 $P_X$ or simply $P$ will mean a constant sheaf of the monoid on $X_{\text{\' et}}$.
 The relative log differential sheaf of $X^\dagger\ra Y^\dagger$
 will be written by $\Omega ^\dagger _{X/Y}$.
 In order to avoid confusion, we often use separated notations $f$ and $\underline{f}$ for
 a log morphism and its underlying morphism, respectively. The symbol $e_i$ is
 the $i$-th element of the standard basis of $\NN ^m$.
 For maps $X\ra Y$ and $Z\ra Y$, the fiber product $X\times _Z Y$
 will be written by $X_{|_Z}$; and by $X_z$ if $Z$ is a point $z$.
We will consider log structures also on algebraic spaces, extending the definition obviously.

\subsection{Acknowledgements} The author thanks
A. Bertram, I. Ciocan-Fontanine, A. Givental, T. Graber, A.
Kresch, J. Li, Y.-G. Oh, R. Pandharipande, B. Siebert, R. Vakil,
and A. Zinger for helpful discussions; M. Olsson for delightful
suggestions; C. Cadman,  Y.-P. Lee, F. Sato and Referee for useful
comments on the paper; and Department of Mathematics, UC-Berkeley
for providing excellent environments while the paper is written
there. This work is partially supported by KRF-2007-341-C00006.

\section{Basics on Logarithmic Geometry}\label{LogGeometry}

Closely following \cite{K.Kato, Ogus},
we recall the basic terminology on log geometry which we use in this paper.

\subsection{Monoids}
A {\em monoid} is a set $P$ with an associative commutative binary
operation $+$ with a unity. We assume that a homomorphism between
monoids preserves the unities. An equivalence relation on a monoid
$P$ is called a congruence relation if $a\sim b$ iff $a+p\sim b+p$
for all $p\in P$. Note that when $\sim$ is a {\em congruence
relation} there is a unique monoid structure on $P/\sim$ such that
the projection $P\ra P/\sim$  is a monoid homomorphism. A monoid
$P$ is called {\em integral} if $p+p_1=p+p_2$ implies that
$p_1=p_2$. A homomorphism $h: P\ra Q$ between integral monoids is
called an {\em integral homomorphism} if for any $h(p_1)+q_1 =
h(p_2) +q_2$, there are $p_1', p_2' \in P,  q \in Q$ such that
$q_1=h(p_1')+q$, $q_2=h(p_2')+q$, and $p_1+p_1'=p_2+p_2'$. The {\em
cokernel} of a homomorphism $h: P\ra Q$ is defined to be the
induced monoid on $Q/P$ where the coset is given as: $q \sim q' $
iff $q+h(p) = q'+h(p')$ for some $p, p'\in P$. The {\em group}
$P^{gp}$ associated to a monoid $P$ is defined to be the induced
monoid on $P\times P /\sim$, where $(p,q) \sim (p',q')$ iff $p+q'+
r= p'+q+ r$ for some $r\in P$. A monoid is called {\em sharp} if
there is no nonzero unit (here units are, by definition, invertible elements).
A nonzero element $u$ in a sharp monoid
is called {\em irreducible} if $u=p+q$ implies that $p$ or $q$ is
zero.

\subsection{Log structures}

A {\em pre-logarithmic} structure on a scheme $X$ is a pair
$(M,\alpha )$ where $M$ is a sheaf of unital commutative monoids
on the \'etale site $X_{\text{\'et}}$ of $X$ and $\alpha$ is a
monoidal sheaf homomorphism from $M$ to $\mathcal{O}_{X}$.  Here $\mathcal{O}_X$ is taken to be a
sheaf of monoid with respect to the multiplications.
When
$\alpha ^{-1}(\mathcal{O}_X ^{\times})\cong
\mathcal{O}_X^{\times}$ via $\alpha$, it is called {\em
logarithmic} structure.
For a given
pre-log structure $(M,\alpha )$, the {\em associated} log
structure $M^a$ is defined to be the amalgamated sum (i.e., the
pushout)
\[  M\oplus _{\alpha ^{-1}(\mathcal{O}_X)}\mathcal{O}_X^\times := M\oplus \co _X^\times /\sim ,\]
where $(m,u)\sim (m',u')$ if there are $a, a'\in \alpha
^{-1}(\mathcal{O}_X^\times) $ for which $a + m = a' + m'$ and $
\alpha (a')u = \alpha (a) u'$.

If $f:X\ra Y$ is a map between schemes and $N$ is a log structure
on $Y$, then define the {\em pullback} $f^*N$ to be the associated
log structure of the prelog structure of $f^{-1}N\ra
f^{-1}\mathcal{O}_Y \ra \mathcal{O}_X$.

A {\em log morphism} $f: (X,M)\ra (Y,N)$ between log schemes consists of a scheme morphism
$f:X\ra Y$ and a monoidal sheaf homomorphism $f^b : f^*N \ra M$ making the diagram
\[\xymatrix{ f^*N  \ar[r] \ar[rd] & M \ar[d]\\
                       & \co _X }\] commute. Later, for simplicity, we often say that
                   $f^b: f^*N\ra M$ is a homomorphism
                   {\em over log structure maps} if the diagram commutes.

A {\em chart} $P_X$ of a log structure $M$ on
$X$ is $\theta : P_X\ra M$, where $P_X$ is the constant sheaf
of a finitely generated integral monoid $P$ making its
associated log structure  isomorphic to the log structure $M$ under $\theta$.
When a chart exits \'etale locally
everywhere on $X$, we say that the log structure is {\em  fine}.
 From now on, we consider only fine log structures.
 We will denote the separable closure
 of $p\in X$ by $\bar{p}$.  The quotient sheaf
 $M/\alpha ^{-1}(\mathcal{O}_X^\times )$, denoted  by $\overline{M}$, is
  called the {\em characteristic}.
 It is useful to note that $\overline{f^*N}=f^{-1}\overline{N}$.
A fine log structure is called {\em locally free} if for every point $p\in X$, $\overline{M}_{\bar{p}}$
 is finitely generated and free,
i.e. $\overline{M}_{\bar{p}}\cong \NN ^r$ for some integer $r$
which possibly depends on $p$. If $M$ is locally free, then for every
$p\in X$, there is a chart $\theta: \NN ^r _X \ra M$ by which $\NN
^r \ra \overline{M}_{\bar{p}}$ is an isomorphism and $M_{\bar p}
\cong \overline{M}_{\bar{p}} \oplus \mathcal{O}_{\bar{p}}^\times$
(Lemma \ref{chart}).

Let $M$ and $M'$ be log structures on $X$ with
$\overline{M}_{\bar{p}}$ and $\overline{M}_{\bar{p}'}$ sharp,
where $p\in X$. A homomorphism $h$ from $M$ to $M'$ will be called
{\em simple} at $p\in X$ if: $\bar{h}: \overline{M}_{\bar{p}}\ra
\overline{M'}_{\bar{p}}$ is injective, and  for any irreducible
element $b \in \overline{M'}_{\bar{p}}$, there is an irreducible
element $a$ in $\overline{M}_{\bar{p}}$ such that $\bar{h}(a)=mb$
for some positive integer $m$, where $\bar{h}$ is the induced map
from $h$. When $M$ and $M'$ are locally free, it means that
$\overline{M}_{\bar{p}}\ra \overline{M'}_{\bar{p}}$ is form of a
diagonal map $d=(d_1,...,d_r) : \NN ^r \ra \NN ^r$, where the
standard basis $e_i$ maps to $d_ie_i$ and $d_i\ne 0$ for all $i$.

A {\em chart of a log morphism} $f$ is a triple $(P_X\ra M, Q_Y\ra
N, Q\ra P)$ of the chart $P$ for $M$, the chart $Q$ for $N$, and a
homomorphism $Q\ra P$ making the diagram
\[\xymatrix{  P _X \ar[r]  & M  \\
                     Q_X \ar[r]  \ar[u] &f^*N \ar[u] }\] commutative.
For any log morphism $f:(X,M)\ra (Y,N)$ and any \'etale local chart $Q\ra N$, there is
a chart  $(P_X\ra M, Q_Y\ra N, Q\ra P)$  of the map $f$ \'etale locally by Lemma 2.10 in \cite{K.Kato}.

 A log morphism $f: (X,M)\ra (Y, N)$ is called:

\begin{itemize}

\item {\em log smooth} if \'etale locally, there is a chart $(P_X\ra M, Q_Y\ra N, Q\ra P)$ of $f$
 such that:
 \begin{itemize}

 \item $\mathrm{Ker}(Q^{gp}\ra P^{gp})$ and the torsion part of $\mathrm{Coker}(Q^{gp}\ra P^{gp})$ are
  finite groups.

\item The induced map $X \ra Y\times _{\Spec (\ZZ [Q])}\Spec \ZZ[P]$ is smooth in the usual sense.

\end{itemize}

\item {\em integral} if for every $p\in X$, the induced map
$\overline{N}_{f(\bar{p})} \ra \overline{M}_{\bar{p}}$ is integral.

\item {\em vertical} if $M/f^*N$ is a sheaf of groups under the induced monoidal operation.

\item {\em strict} if $f^b: f^*N \ra M$ is an isomorphism.

\end{itemize}

All above notions are preserved under base changes.

The amalgamation of integral monoids

\[ \xymatrix{      &   P_1 \\
                      P_2 & Q \ar[u]_{\theta} \ar[l] }\]

is not longer integral in general. However, when
$\theta$ is integral, the amalgamation is always integral.
Hence in that case, the base change, in other words, the fiber
product of an integral log morphism $f_1$ is defined to be
\[ (X_1,M_1)\times _{(Y,N)} (X_2,M_2) := (X_1\times _Y X_2, (P_1\oplus _Q P_2 )^a )\]
where $P_1, P_2, Q$ are charts of the log morphisms $f_i:
(X_i,M_i)\ra (Y,N)$. See Proposition 4.1 in \cite{K.Kato} for
various equivalent definitions of integral morphisms. For
instance, we see that the underlying map of a smooth and integral
morphism is flat (Corollary 4.5 \cite{K.Kato}).

\section{Log Curves}

This section deals with the definition and some properties of
 the allowed domains of log stable maps. The domains
will be minimal log prestable curves defined in subsection \ref{MinimalCurve}.
We start with fixing a notion of log prestable curves basically following
\cite{Mo, We, F.Kato_LogSmoothCurve, Olsson_UnivLog}
for our purpose.

%
%

\subsection{1st definition}\label{logprecurve1} A log morphism $\pi: (C,M) \ra (S,N)$ is
called a {\em log prestable curve} over a fine log scheme $(S,N)$
if every geometric fiber of $\underline{\pi}$ is a prestable curve
and $\pi$ is a proper, log smooth, integral, and vertical morphism.

\medskip

According to \cite{F.Kato_LogSmoothCurve} and
\cite{Olsson_UnivLog}, this definition is equivalent to the
following second definition.

%
%

\subsection{2nd definition} \label{logprecurve2} A log morphism $\pi: (C, M)\ra (S, N)$
 is called a {\em log prestable curve} over $(S,N)$ if:
\begin{enumerate}

\item The underlying map $\underline{\pi} : C\ra S$ is a $S$-family of prestable curves.

\item
\begin{itemize} \item The restriction of $\pi ^* N \ra M$
to the $\underline{\pi}$-smooth locus $(C/S)^{\mathrm{sm}}$ is an isomorphism.

\item If $p$ is a singular point with respect to $\underline{\pi}$,
          then there are an \'etale neighborhood $U$ of $\bar{p}$,
          an affine \'etale neighborhood $\Spec A$ of
           $\pi (\bar{p})$ in $S$, and  a chart
           $((\NN ^2 \oplus _{\NN} Q)_C\ra M, Q_S\ra N, Q\ra \NN ^2 \oplus _{\NN} Q)$
           of $\pi$ such that
           the induced map $$U\ra \Spec( A \ot  _{\ZZ[Q]}\ZZ [\NN ^2 \oplus _{\NN} Q])$$
           is \'etale, where $\NN \ra \NN^2$ is
           the diagonal $\Delta$, sending $e_1\mapsto e_1+e_2$.

\end{itemize}

\item The log structure of $N$ is fine.

\end{enumerate}

\subsection{Remarks}

Here are some remarks on the definitions above.

\subsubsection{}

Since the homomorphism $\Delta$ and the monoid $Q$ are integral, the monoid $\NN ^2\oplus _{\NN} Q$
and the homomorphism $\NN ^2\oplus _{\NN} Q\ra Q$ are also integral by Proposition 4.1 in \cite{K.Kato}.

\subsubsection{}  Let  $Q\ra N$ be a chart at $s:=\pi (\bar{p})$, which induces
an isomorphism $Q\ra \overline{N}_{\bar{s}}$. Such a chart is called good at $s$ in the past literature.
Then at any node $\bar{p}$ of $C_{\bar{s}}$, \'etale locally $f$ has a
chart $((\NN ^2\oplus _{\NN} Q)_C \ra \co _C, Q_S \ra \co _S,  Q \ra\NN ^2\oplus _{\NN} Q)$ satisfying:
 the induced map $\mathcal{O}_S\ot _{\ZZ [Q]} \ZZ [ \NN ^2\oplus _{\NN} Q]\ra \mathcal{O}_C$ is \'etale, and
 the induced map $  \NN ^2\oplus _{\NN} Q \ra \overline{M}_{\bar{p}}$ is an isomorphism.
  We  can see this by working on the characteristics $Q=\overline{N}_{\bar{s}}$ and $P=\overline{M}_{\bar{p}}$,
  as following.
 By the second definition of log prestable curves and the statment 1 in Lemma \ref{chart} below,
 we know that $P$ is the amalgamation sum of
 \[ \xymatrix{ \NN \ar[rr]_{e_1\mapsto e_1+e_2} \ar[d] & & \NN ^2 \\
                       Q . & &                     }\]
                        Now we lift $P$ and $Q$ to make
 charts for $M$ and $N$, using 2 and 3 in Lemma \ref{chart}.

 \medskip

\begin{Lemma}\label{chart}  Let $(M,\alpha )$ be a fine log structure on a scheme $X$
 and  let $x \in X$.

1.  A homomorphism $\theta: P\ra M_{\bar{x}}$ induces a chart \'etale locally at $\bar{x}$ if and only if
     the induced homomorphism
     $P/(\alpha\circ\theta)^{-1}(\co _{\bar{x}}^\times) \ra \overline{M}_{\bar{x}}$
     is an isomorphism.

2. Suppose that $\theta : \overline{M}_{\bar{x}} \ra M_{\bar{x}}$ is a
    lift of $M_{\bar{x}}\ra \overline{M}_{\bar{x}}$,
    then $\theta$ provides a local chart of $M$ at ${\bar{x}}$.

3.  {\em (Proposition 2.1 \cite{Olsson_LogGeom})} Let $(X,M)$ be a
fine log scheme, and let $x\in X$ such that the characteristic of
the residue field of $x$ is zero.
 There is a lift of $M_{\bar{x}} \ra \overline{M}_{\bar{x}}$.

\end{Lemma}

\begin{proof}
1. This is alluded in  \cite{K.Kato}. The \lq\lq only if\/" part  is clear. We prove the \lq\lq if\/" part.
It amounts to showing the natural homomorphism
 \[ P \oplus _{(\alpha\circ\theta)^{-1}(\co _{\bar{x}}^\times )} \co ^\times _{\bar{x}} \ra M_{\bar{x}} \]
 is an isomorphism.
 It is clearly surjective. For 1-1, let $\theta (p) +u = \theta (p') + u'$, where
 $(p,u), (p', u') \in M_ {\bar{x}}\times \co _{\bar{x}} ^\times$.
 Then $\overline{\theta (p)} = \overline{\theta (p')} $ in $\overline{M}_{\bar{x}}$ and so by assumption
 \begin{equation}\label{1} p+v = p'+v' \end{equation}
  for some
 $v, v' \in (\alpha\circ\theta)^{-1}(\co _{\bar{x}}^\times)$. Here we abuse notation.
 Thus, $\theta (p)+v = \theta (p') +v'$. This together with
  $\theta (p) +u = \theta (p') + u'$ implies that $\theta (p)+v + u' = \theta (p') + v' +u' =
  \theta (p)+u + v'$, which in turn shows that
  \begin{equation}\label{2} u'+v = u+v'. \end{equation}
  By \eqref{1} and \eqref{2}, $(p,u)\sim (p',u')$.

  2. This follows from 1 since $(\alpha\circ \theta)^{-1}(\co _{\bar{x}} ^\times)=\{ 0\} $.
   \end{proof}

\subsection{} For a prestable curve $C/S$, there is a canonical log structure $M^{C/S}$ on $C$
(resp. $N^{C/S}$ on $S$) induced from the log structure on the stacks $\fM _{g,1}$ (resp. $\fM _g$)
of genus $g$, $1$-pointed (resp. no-pointed) prestable curves (\cite{Mo, F.Kato_LogSmoothCurve}).
Here the log structures on a smooth cover is given by
the boundary divisors of singular curves
and the log structures are defined on smooth topology (\cite{Olsson_LogGeom}).
Thus, the natural log morphism $(C,M^{C/S})\ra (S,N^{C/S})$ is a log prestable curve.
More generally, for a given homomorphism $N^{C/S}\ra N$  from the canonical log structure $N^{C/S}$
to a fine log structure $N$ on $S$,
 the amalgamated sum $M$ defined to be the log structure associated to
the prelog structure \[ M^{C/S}\oplus _{\pi^{-1}N^{C/S}} \pi^{-1}N \]
yields a  log prestable curve over $(S,N)$. All log prestable curves are obtained
 in this manner according to Proposition 2.3 (and Remark 2.4) in \cite{F.Kato_LogSmoothCurve}
 and  Theorem 2.7 in \cite{Olsson_UnivLog}.

\subsection{3rd definition} A pair $(C/S, N^{C/S} \ra N)$ is called a {\em log prestable curve} over $(S,N)$ if
$C/S$ is a prestable curve over $S$, and $N^{C/S}\ra N$ is a
homomorphism from the canonical log structure $N^{C/S}$ on $S$,
induced from the family $C/S$, to a fine log structure $N$ on $S$.

\subsection{Special coordinates}\label{SpecialCoord}

   Let $C/S$ be a prestable curve; let $p$ be a nodal point;
   $A:=\co _{\pi( \bar{p})}$; and
                            $R:=\mathcal{O}_{\bar{p}}$ be
                             the strict henselianization
                           of $A[u,v]/(uv- t)$ at the ideal $(u,v,\mathfrak{m}_A)$,
                           where $t\in\mathfrak{m}_A$.

   Then we  call $(u,v)$  a {\em special coordinate pair}
   at the node $p$ defined by the ideal $(u,v, \mathfrak{m}_A)$.
   Based on the coordinates, in fact, the canonical log structure is constructed.
   Locally define prelog structures $\NN ^2 \ra \mathcal{O}_C$ by
   sending $e_1\mapsto u$, $e_2\mapsto v$ and define $\NN \ra \mathcal{O}_S$ by sending $e_1\mapsto uv$.
   Since by Lemma \ref{SpecialLemma}
    special coordinates are unique up to multiplications by unique elements in $\mathcal{O}_C^\times$ whose
   product is in $\mathcal{O}_S^\times$, such prelog structures are isomorphic up to unique isomorphism. They
   can be glued together. See \cite{F.Kato_LogSmoothCurve} for the detail of the construction of the canonical
   log structure on a (pre)stable curve $C/S$.

    \begin{Lemma}\label{SpecialLemma}  {\em (\cite{F.Kato_LogSmoothCurve})}
  Provided with the notation as in the beginning of \ref{SpecialCoord}, we have:

 1. Let $u'$, $v'$ be elements in $R$. Suppose that the ideals
 $(u',v', \mathfrak{m}_A)$ and $(u,v, \mathfrak{m}_A)$ coincide, and the product $u'v'$ is an element in $A$.
          Then there are  $a, b \in R^{\times}$ such that  $u'=au$, $v'=bv$
          (or $u'=av$, $v'=bu$) and $ab\in A^\times$.

 2. Suppose that $u^l = au^l$ and $v^l = bv^l$, where $l\in \NN _{\ge 1}$, $a, b \in R^{\times}$. If
     $ab\in A^\times$, then
     $a=b=1$.

 \end{Lemma}

 \begin{proof}
 The first statement and the second statement with $l=1$ are
 proven in \cite{F.Kato_LogSmoothCurve}. The proof of Lemma 2.2 in
 \cite{F.Kato_LogSmoothCurve} for $l=1$ works also for the general $l$ since $R\subset \hat{R}$, where
 $\hat{R}$ is the $t$-adic completion of $R$.
 \end{proof}

     For example, in the localization of $\Lambda [u,v]/(uv)$ at the ideal $(u,v)$,
  let $u'=(1+u)u$, $v'=(1+v)v$. Then we have $u'=(1+u-\frac{v}{1+v})u$
      and $v'=(1+v-\frac{u}{1+u})v$, where the product $(1+u-\frac{v}{1+v}))(1+v-\frac{u}{1+u}) =1$.

   %
    %

 \begin{Cor}\label{LoguLogv} Let $l$ be a positive integer and let $\pi: (C,M)\ra (S,N)$ be
 a log prestable curve.  Then with the notation as in the beginning of subsection \ref{SpecialCoord},
in $M_{\bar{p}}$ there is a unique pair $\gamma _u$, $\gamma _v$
--- which will be denoted by $l\log u$ and $l\log v$, respectively
--- such that $\gamma _u +\gamma _v  \in  N_{\pi (\bar{p})}$ and
$\alpha (\gamma _u)=u^l$, $\alpha (\gamma _v)=v^l$, where $\alpha$
is the structure map $M_{\bar{p}}\ra \co _{\bar{p}}$.
\end{Cor}

\begin{proof} The existence follows from the second definition of log prestable curves.
If $(\gamma _u',\gamma _v')$ is another such pair. Then
it is clear that $\gamma _u'=\gamma _u$ and $\gamma _v'=\gamma _v$ in
$\overline{M}_{\bar{p}}$. Hence, $\gamma _u'=\gamma _u+c_u$, $\gamma _v'=\gamma _v+c_v$, where
$c_u, c_v \in \co _{\bar{p}}^\times $. Since $\gamma _u'+\gamma _v' , \gamma _u+\gamma _v \in
N_{\pi(\overline{p})} $ and $\gamma _u'+\gamma _v'=\gamma _u+\gamma _v$ in
$\overline{N}_{\pi(\overline{p})}$, we see that $c_uc_v \in \co
_{\pi (\bar{p})}^\times$. Now we apply Lemma \ref{SpecialLemma},
to conclude the proof.
\end{proof}

\subsection{A minimal log prestable curve}\label{MinimalCurve}

We call a log prestable curve  $$(C/S,N^{C/S}\ra N)$$ minimal in
weak sense (\cite{We}) when $N$ is locally free and there is no proper
{\em locally free} submonoid of $N$ containing the image of $N^{C/S}$.
This amounts that for all $s\in S$, $\overline{N}_{\bar{s}}$ is a
finitely generated free monoid and there is no proper free
submonoid of $\overline{N}_{\bar{s}}$ containing \[
\mathrm{Im}(\overline{N}^{C/S}_{\bar{s}}\ra
\overline{N}_{\bar{s}}) . \] It is called {\em minimal} in strong
sense furthermore if for every irreducible $b\in \overline{N}_{\bar{s}}$,
there is an irreducible element $a\in
\overline{N}_{\bar{s}} ^{C/S}$ such that $a=lb$ for some positive
integer $l$. Since we will use only \lq minimal in strong sense' in this paper,
by \lq minimal', we will mean \lq minimal in strong sense' unless otherwise
stated.

\subsection{Remark}
  It is straightforward to check
that the fibered categories $\mathfrak{M}_g^{\mathrm{log}}$ of log
prestable curves in either definitions, 1st one and 3rd one, are
equivalent as following. Let us choose pullbacks once for all and
then identify them, for example $(f\circ g)^* M = g^*(f^* M)$, if
there exist canonical isomorphisms. A morphism, i.e., arrow,
$(C',M')/(S',N') \ra (C,M)/(S,N)$ is a pair $(h, \phi )$, where $h:
(C',M') \ra (C,M)$ and $\phi : (S',M')\ra (S, M)$ are morphisms
between log schemes such that:
\begin{itemize}
\item  The underlying scheme morphisms provide a cartesian square diagram
\begin{equation}\label{CatesianCurve} \begin{CD} C' @>>{\underline{h}}> C \\
                 @VV{\underline{\pi} '}V @VV{\underline{\pi}}V \\
                   S' @>>{\underline{\phi}}> S . \end{CD}\end{equation}
\item  $h^b: h^*M\ra M'$ and $\phi ^b: \phi ^*N\ra N'$ are isomorphisms making the diagram
\[ \begin{CD} M' @<<< \underline{h}^*M  \\
                      @AAA      @AAA \\
                         \underline{\pi} ^* N'  @<<<
                         \underline{h}^*\underline{\pi} ^*N=(\underline{\pi} ')^*\underline{\phi} ^*N
                      \end{CD}\]
commute over log structure maps.
\end{itemize}

In view point of the third definition, a morphism  $(C'/S',
N^{C'/S'}\ra N')\ra (C/S, N^{C/S}\ra N)$ is a pair
$(\underline{h}:C'\ra C, \phi : (S',N')\ra (S,N))$ such that
$\underline{h}$ and $\underline{\phi}$ give rise to a cartesian square
diagram of underlying schemes as in \eqref{CatesianCurve}  and the
diagram
\[ \begin{CD} \underline{\phi} ^*N @>>_{\phi ^b}> N' \\
                       @AAA        @AAA \\
                       \underline{\phi} ^* N^{C/S} @>>{=}> N^{C'/S'}
     \end{CD}\]
which commutes over log structure maps. Here the equality means
the second part in the pair of the canonical isomorphisms $(\underline{h}^*M^{C/S},
\underline{\phi} ^*N^{C/S}) \ra (M^{C'/S'}, N^{C'/S'})$ between
two pairs of canonical log structures on $C'/S'$. We will see
later that the fibered category of log prestable curves
 is an algebraic stack over the category $(\mathrm{Sch}/\Lambda)$ of
schemes over $\Lambda$.


\section{Log twisted FM type spaces}

This section deals with
 the allowed targets of log stable maps, called
log twisted FM type spaces.

\subsection{(log) FM type spaces}\label{FMtypespaces}

 An algebraic space $W$ over a scheme $S$ is called a  {\em FM type space}
if  $W\ra S$ is a projective, flat map whose geometric fibers are
semi-stable varieties of form $xy=0$, i.e., for every $s\in S$,
\'etale locally, there is an \'etale map $W_{\bar{s}} \ra \Spec
k(\bar{s})[x,y,z_1,...,z_{r-1}]/(xy)$ where $x,y, z_i$ are
independent variables with only one relation $xy=0$.

Furthermore, if $W/S$ allows a special log morphism, then we call it
a {\em log FM type space}.

 \bigskip

Here we say that for a FM type space $W/S$, a log smooth morphism
$$\pi : (W,M^{W/S}) \ra (S,N^{W/S})$$ is {\em special} (\cite{Olsson_UnivLog})  if:
\begin{itemize}

\item $M^{W/S}$ and $N^{W/S}$ are locally free.

\item For any $w\in W$, the induced map $\overline{\pi ^{*} N}^{W/S}_{\bar{w}} \ra \overline{M}^{W/S}_{\bar{w}}$
 is either an isomorphism or the part
of the cocatesian diagram
\[ \xymatrix{ \NN \ar[r]^{\Delta}_{e_1\mapsto e_1+e_2}\ar[d]_{h} & \NN ^2\ar[d] \\
                                      \overline{N}^{W/S}_{\pi(\bar{w})} \ar[r] & \overline{M}^{W/S}_{\bar{w}},
}\] where $h(e_1)$ is an irreducible element.

\item The natural map $W^{\mathrm{sing}}_{\bar{s}} \ra \mathrm{Irr} \overline{N}_{\bar{s}} $
   induced by the above diagram gives rise to a bijection
         $$\mathrm{Irr}   W^{\mathrm{sing}}_{\bar{s}} \ra \mathrm{Irr}\overline{N}_{\bar{s}}, $$ where
         $ \mathrm{Irr}   W^{\mathrm{sing}}_{\bar{s}} $ is the set of irreducible components of
 the singular locus $W^{\mathrm{sing}}_{\bar{s}}$ of $W_{\bar{s}}$ and
$\mathrm{Irr} \overline{N}_{\bar{s}} $  is the set of irreducible elements of $\overline{N}_{\bar{s}}  $.

\end{itemize}

We will also call the special log structure the {\em canonical} one
and denote it by $M^{W/S}$ and $N^{W/S}$, respectively
as in the definition above.

 \subsection{}

As a generalization of log twisted curves
(\cite{Olsson_LogTwisted}) to a FM type space $W/S$ we define the
following.

A {\em log twisted FM type space} is
a pair $(W/S, N^{W/S}\ra N)$, where $W/S$ is a log FM type space
and $N^{W/S} \ra N$ is a {\em simple} map from the canonical log structure $N^{W/S}$ to a
locally free log structure $N$ of $S$.
By Lemma \ref{chart}, this amounts that \'etale locally
at $s\in S$ there is a commuting diagram of charts:
\begin{equation}\label{LogTwisting} \xymatrix { N^{W/S} \ar[r] & N   \\
                    \NN ^{m} \ar[r] \ar[u]_{\theta ^{W/S}}& \NN ^{m} \ar[u]_{\theta},
                    }\end{equation}
where the bottom map is a diagonal map between the constant sheaves $\NN ^m$
                    and $\theta ^{W/S}$ (resp. $\theta$) induces an isomorphism
                    from $\NN ^{m}$ to $\overline{N}^{W/S}_{\bar{s}}$
                    (resp. $\overline{N}_{\bar{s}}$).

\subsection{}\label{ExtendedLogTwistedSpace}
An {\em extended log twisted FM type space} is a pair $(W/S,
N^{W/S}\ra N)$, where $W/S$ is a log FM type space and $N^{W/S} \ra
N$ is an extended simple map from the canonical log structure
$N^{W/S}$ to a log structure $N$ of $S$. This means, as the
definition, that \'etale locally there is a commuting diagram of
charts:
\begin{equation}\label{extendedLogTwisting} \xymatrix { N^{W/S} \ar[rr] & & N   \\
                    \NN ^{m} \ar[r] \ar[u] & \NN ^{m}  \ar[r] & \NN ^{m}\oplus \NN ^{m'}, \ar[u]
                    }\end{equation}
where the first bottom map is a diagonal map, the second bottom
map is the natural monomorphism $(\mathrm{id}_{\NN ^m},0)$, and
vertical maps induce isomorphisms between $\NN ^m$ (resp. $\NN
^m\oplus \NN ^{m'}$) and $\overline{N}^{W/S}_{\bar{s}}$ (resp.
$\overline{N}_{\bar{s}}$). In particular, $N$ is a locally free
log structure on $S$. We often write simply $(d,0): N^{W/S} \ra N$
for (\ref{extendedLogTwisting}) with the diagonal map being
$d=(d_1,...,d_m)$.

We endow a log structure $M$ on $W$ by the amalgamated sum
\[ (M^{W/S} \oplus _{\pi ^{-1}N^{W/S}} \pi ^{-1}N)^a.\]
Conversely, from Theorem 2.7 of \cite{Olsson_UnivLog} we arrive at another equivalent definition.

\subsection{}
An {\em  extended log twisted FM type space} is a proper, log
smooth, integral, vertical morphism $\pi :(W, M)\ra (S,N)$ such
that the underlying map $\pi : W\ra S$ is a FM space; the log structure $N$ is locally free; $\pi ^b$ is
an isomorphism on the smooth locus of $W/S$; and at a singular
point of $W/S$, \'etale locally $\pi$ has a chart:
\[\xymatrix{  \pi^{-1} N  \ar[rr] & &  M \\
 Q:=\NN ^{r-1}\oplus \NN  \ar[rr] _{(\mathrm{id},\phi)} \ar[u] & & P:=\NN ^{r-1}\oplus B \ar[u] \\
                                            }\]
                           where the monoid $B$ is the amalgamated sum in
\[\begin{CD}  \NN @>{e\mapsto e_1+e_2}>{\Delta}>  \NN ^2\\
                        @V{\times d}VV @VVV \\
                     \NN   @>{\phi}>> B \end{CD};\]
the induced map $\mathcal{O}_S\ot _{\ZZ[Q]} \ZZ[P]\ra
\mathcal{O}_W$ is smooth; and
for every $s\in S$,
the natural map $W^{\mathrm{sing}}_{\bar{s}} \ra \mathrm{Irr} \overline{N}_{\bar{s}} $
induced by the above diagram gives rise to an injection
$$\mathrm{Irr}   W^{\mathrm{sing}}_{\bar{s}} \ra \mathrm{Irr}\overline{N}_{\bar{s}}. $$

\subsection{Log twisted FM spaces}
For a nonsingular projective variety $X$ over $\bf k$, let
 $X[n]$ be the Fulton-MacPherson configuration space of $n$
labeled distinct points in $X$
 and let $X[n]^+$ be its universal space (\cite{FM}).

 In \cite{KKO}, a Fulton-MacPherson degeneration space or in short, a {\em FM space} $W$ of $X$ over $S$
 is defined to be a pair $(W\ra S, \pi _X:W\ra X)$ of maps, where: $W$ is an algebraic space over a scheme $S$; and
  for every $s\in S$, there are an \'etale neighborhood $T$ of $\bar{s}$, and a cartesian diagram
  \[ \xymatrix{ W_{|_T} \ar[r] \ar[d] & X[n]^+ \ar[d]\\
                          T \ar[r] & X[n]
                          }\] for some $n$
   such that, through the diagram, $\pi _X$ is compatible with the composite
    \[\xymatrix{ X[n]^+\ar[r] & X^{n+1}\ar[rr] _{\mathrm{pr}_{n+1}} & &X. } \]

By the blowup construction of the universal family $X[n]^+$ from
$X[n]\times X$, it is clear that $X[n]^+\ra X[n]$ is a log FM type
space, and thus FM spaces are log FM type spaces.

 There are two more examples constructed in \cite{KS}: $(X_D^{[n]})^+/X_D^{[n]}$
 (the configuration space of $n$ labeled points away from a smooth closed subvariety $D$ in $X$)
and $X_D[n]^+/X_D[n]$ (the configuration space of $n$ labeled
distinct points away from a smooth closed subvariety $D$ in $X$).
The constructions in \cite{KS} are valid over any closed field $\bf k$
instead of $\CC$, without any changes.


\section{Log stable maps}\label{LogStableMaps}

In this section, we define  log stable maps and study their properties.
 We first need to explain what the underlying maps of
 the log stable maps are.

\subsection{Admissible maps and stable unramfied maps}

We recall some definitions in \cite{Li1, KKO}.

\subsubsection{}\label{adm}
A triple
\begin{equation}\tag{$\star$}\label{star} ((C/S, {\bf p}), W/S, f: C \ra W)     \end{equation}
 is called a $n$-pointed, genus $g$, {\em admissible map}
to a FM type space $W/S$ if:

\begin{enumerate}

\item $(C/S, {\bf p}= (p_1,...,p_n))$ is a $n$-pointed, genus $g$, prestable curve over $S$.

\item $W/S$ is a FM type space.

\item $f:C\ra W$ is a map over $S$.

\item ({\em Admissibility}) If a point $p\in C$ is mapped into the relatively singular locus $(W/S)^{\mathrm{\sing}}$ of $W/S$,
then \'etale locally at $\bar{p}$, $f$ is factorized as
\[ \xymatrix{ C  \ar[dd]_{f}\ar[dr]&         & U\ar[ll]\ar[rr] \ar[rd] \ar[dd] &    & \Spec(A[u,v]/(uv- t)) \ar[ld]  \ar[dd]\\
                                                         & S   &     & \Spec A \ar'[l][ll]  &                          \\
                       W   \ar[ur] &       &   V \ar[ll]\ar[rr] \ar[ur]&     & \Spec A[x,y, z_1,...,z_{r-1}]/(xy-\tau ) \ar[ul]
      }\]
     where all 5 horizontal maps
 are formally \'etale; $u, v, x, y, z_i$ are indeterminates;
  $x=u^l$, $y=v^l$ under the far right vertical map for some positive integer $l$;
 $t, \tau$ are elements in the maximal ideal $\mathfrak{m}_A$ of
 the local ring $A$; and $\bar{p}$ is mapped to the point defined by the ideal $(u,v,\mathfrak{m}_A)$.

\end{enumerate}

\medskip

We call a node $\bar{p}$ of $C_{\bar{s}}$ {\em distinguished}
(resp. {\em nondistinguished}) if $f(\bar{p})\in
W^{\mathrm{sing}}_{\pi(\bar{p})}$ (resp. $f(\bar{p})\in
W^{\mathrm{sm}}_{\pi (\bar{p})}$), where $\pi$ is the
map $C\ra S$.

\subsubsection{Remark} The admissibility can be stated as below, too.
                            If $p\in f^{-1}((W/S)^{\mathrm{\sing}})$,   then under $f$,
                            $x=c_1u^l$, $y=c_2v^l$ for some $c_i \in R^\times$ such that $c_1c_2 \in A^\times$
                            where $A:=\co _{\pi( \bar{p})}$;
                            $R:=\mathcal{O}_{\bar{p}}$ is  form of
                             the strict henselianization
                           of $A[u,v]/(uv- t)$ at the ideal $(u,v,\mathfrak{m}_A)$;  $\mathcal{O}_{f(\bar{p})}$ is form
                           of the strict henselianization of $A[x,y, z_1,...,z_{r-1}]/(xy-\tau )$
                           at the ideal $(x,y, z_1,...,z_{r-1}, \mathfrak{m}_A)$;
                           and $t, \tau$ are in the maximal ideal  $\mathfrak{m}_A$ of $A$.

                            Two definitions are equivalent: One direction is clear. We prove the other direction.
                           Consider polynomials
                           $w^l - c_i$ in $R[w]$ where $w$ is an indeterminate. Each polynomial has a linear coprime
                           factorization over $R/\mathfrak{m}_R$. Since $R$ is a henselian ring, the linear
                           factorization can be lifted over $R$, providing a solution $c_i^{1/l} \in R^\times $ to  $w^l=c_i$.
                           Now we have $x=(u')^l$ and $y=(v')^l$, where $u'=c_1^{1/l}u$, $v'=c_2^{1/l}v$.
                           On the other hand, note that $u'v'\in A$ since $(u'v')^l=xy=\tau \in A$.  Therefore, we can apply
                           Lemma \ref{SpecialLemma} to replace
                           $c_i^{1/l}$ by  elements $a_i \in R^\times$
                           such that $a_1a_2 \in A^\times $. Finally we conclude that $u'v' \in \mathfrak{m}_A$.

 \subsubsection{}
 Let a FM type space $W/S$ be equipped with a map $\pi _X: W\ra X$ from $W$ to a scheme $X$.
An admissible map $(\star)$ is called a  {\em stable admissible
map} to a FM type space $W/S$ of $X$ if the {\em stability with
respect to $\pi _X$} holds, namely: The automorphism group
$\mathrm{Aut} _X(f_{\bar{s}})$  is finite for all $s\in S$, where $\mathrm{Aut} _X(f_{\bar{s}})$ consists of
  all pairs $(h,\varphi )$
    of automorphisms $h$ of $C_{\bar{s}}$ preserving
     $n$-labeled points  and automorphisms $\varphi$ of
       $W_{\bar{s}}$ with respect $X$ such that they are compatible with
       $f_{\bar{s}}$,
     i.e., $h({\bf p}_{\bar{s}})={\bf p}_{\bar{s}}$, $\pi _X\circ \varphi = \pi _X$,
     and $f_{\bar{s}}\circ h = \varphi \circ f_{\bar{s}}$.

\subsubsection{}\label{RamifiedMaps} Let $\mu =(\mu _1,...,\mu _n) \in \NN ^n_{\ge 1}$.
A stable admissible map   $(\star)$ is called a  {\em stable
$\mu$-ramified map} to a FM space of a smooth projective scheme
$X$ over $\bf k$ if:

\begin{itemize}

\item $W/S$ is a FM space of $X$.

\item For all $s\in S$, $f(p_i)_{\bar{s}}$ are pairwise distinct for all $i$.

\item ({\em Strong Stability}) For all $s\in S$, $\mathrm{Aut}_X (f_{\bar{s}})$ is a finite group.
Furthermore, every end component of $W_{\bar{s}}$ contains either a non-line image of
an irreducible component of $C_{\bar{s}}$ or the images of at least two labeled points.

\item $f((C/S)^{\sing})\subset (W/S)^{\sing}$ and $f$ is unramified everywhere
 on the relatively smooth locus $(C/S)^{\sm}$, possibly except at the labeled points.

 \item The ramification order at $p_i$ is exactly $\mu _i$.


\end{itemize}

Here the end components of $W_{\bar{s}}$ are the screen components
which correspond the end nodes of the dual graph of $W_{\bar{s}}$.
See \cite{KKO} for the precise definition.

\medskip

When $\mu =(1,...,1)$, we call the map a {\em stable unramfied
map}. Here we follow the ramification order convention that
ramification order $2$ means the simple ramification.

\subsection{Log stable maps} Combining all previous notions,
we introduce a series of definitions.

\subsubsection{}  A pair $((C,M)/(S,N), {\bf p})$ is called a $n$-pointed, genus $g$,
   (resp. minimal) log prestable curve over $(S,N)$ if
$(C,M)/(S,N)$ is a genus $g$ (resp. minimal) log prestable curve
and $(C/S,{\bf p})$ is a $n$-pointed prestable curve over $S$.

\subsubsection{}\label{LogPrestable}  A log morphism
\begin{equation}\tag{$\star\star$}\label{twostars}    \left( f: (C,M_C, {\bf p})\ra
(W, M_W)\right) /(S,N) \end{equation}
      is called a $(g,n)$ {\em log prestable map over $(S,N)$}
       if:

\begin{enumerate}

    \item $((C,M)/(S,N), {\bf p})$ is a $n$-pointed, genus g, minimal log prestable curve.

    \item $(W, M_W) /(S,N)$ is an extended log twisted FM type space.

    \item  ({\em Corank = \# Nondistinguished Nodes Condition})
              For every  $s\in S$, the rank of $\mathrm{Coker}(N^{W/S}_{\bar{s}} \ra N_{\bar{s}})$ coincides with
                the number of nondistinguished nodes on $C_{\bar{s}}$.

    \item $f: (C,M_C)\ra (W, M_W)$ is a log morphism over $(S,N)$.

    \item\label{LogAdmissible} ({\em Log Admissibility}) either of the following conditions, equivalent
    under the above four conditions, holds:

    \begin{itemize}

    \item $\underline{f}$ is admissible.

    \item $f^b: f^*M_{W}\ra M_C$ is simple at every distinguished node.

    \end{itemize}

   \end{enumerate}



\subsubsection{Log Admissibilty}\label{description}

We want to see the explicit meaning of  the last condition
(\ref{LogAdmissible}) in \ref{LogPrestable} under the rest
conditions imposed. Provided with the notation in the admissible
condition in \ref{adm},
        there exist $\log x$ and $\log y$ in $(M_W)_{f(\bar{p})}$
    such that $\log x + \log y \in N_{\bar{s}}$,
     $\alpha _W(\log x) = x$, and $\alpha _W (\log y)=y$,
    where $\bar{p}\mapsto \bar{s}$ under $C\ra S$, and $\alpha _W: M_W\ra \mathcal{O}_W$
    is the log structure map.
      Then $f^b(\log x)$ and $f^b(\log y)$ must be the $l\log u$ and the $l\log v$, respectively, since
     the pair $(f^b(\log x), f^b(\log y))$  satisfies the assumption in Corollary \ref{LoguLogv}.
Therefore, at a distinguished node $p$, the
log prestable map is described as: at chart levels
there is a diagram
\[ \left(\begin{array}{c}
         \xymatrix{
       \NN ^2\oplus _{\NN }\NN \ar[rrrr] ^{\log x,\ \log y \ \ \mapsto \ \ l\log u,\ l\log v}
       & &  & & \NN ^2 \oplus _{\NN} \NN \\
         &  & \NN   \ar[ull] ^{e_1\mapsto (0,e_1)}
             \ar[urr] _{e_1\mapsto  (0, e_1)}&  & }\end{array}\right) \oplus \NN ^{m+m'-1} \]
               for some positive integer $l$, where
               LHS and RHS amalgamated sums are given by
               \[ \begin{array}{lcr} \xymatrix{ \NN \ni e_1\  \  \  \  \ar@{|->} [r] \ar@{|->}[d] & de_1\in \NN \\
                                     \NN ^2 \ni \  e_1+e_2 &            }   &; &
                                     \xymatrix{ \NN \ni e_1\  \  \  \  \ar@{|->} [r] \ar@{|->}[d] & \Gamma e_1\in \NN \\
                                     \NN ^2 \ni \  e_1+e_2 &            }   \end{array}
                                        \]
               for some positive integers    $d, \Gamma $, respectively.

\medskip

Note also that there is a natural map $\mathrm{Irr} \overline{N}
^{C/S}_{\bar{s}}\ra \mathrm{Irr} \overline{N}_{\bar{s}}$ by
sending $a$ to $b$ if $\Gamma b=a$ under the map
$\overline{N}^{C/S}_{\bar{s}}\ra \overline{N}_{\bar{s}}$ for some
positive integer $\Gamma$.

 \medskip

  Summing these observations together, we claim that if $f$ is a log prestable map then,
      \'etale locally at every geometric point $\bar{s}\ra S$, there are commuting charts
         \begin{equation}\label{DomainTarget}\xymatrix{
   \NN ^{m''}\oplus \NN ^{m'} \ar[d] \ar[rr]_{(\Gamma, \id)}   &  & \ar[d]\NN ^{m}\oplus \NN ^{m'}
   &  & \ar[d]\ar[ll]^{(d,0)} \NN ^m \\
                                  N^{C/S} \ar[rr]\ar[d] & & N\ar[d] & & \ar[ll]\ar[d] N^{W/S} \\
                            \overline{N}^{C/S} \ar[rr] & & \overline{N} & & \ar[ll] \overline{N}^{W/S}
                            }\end{equation}
such that the vertical maps induce  isomorphisms between chart
monoids and the characteristics at $\bar{s}$, where $m''=$ the
number of distinguished nodes in $C_{\bar{s}}$; $m'=$ the number
of nondistinguished nodes in $C_{\bar{s}}$; $m=$ the number of
irreducible components of $W_{\bar{s}}^{\mathrm{sing}}$; $\Gamma $
is a \lq generalized diagonal' matrix of size $m\times m''$
\[ \Gamma = \left( \begin{array}{c|c|c|c}
          \Gamma _{1,1} \cdots \Gamma _{1,k_1} & 0 \cdots 0 & \cdots  & 0 \cdots 0 \\  \hline
        0    \cdots 0       & \Gamma _{2,1} \cdots \Gamma _{2,k_2} & 0 \cdots & 0 \cdots  0 \\ \hline
        \vdots  &                   &      \ddots          & \vdots \\ \hline
        0 \cdots 0 &      \cdots         &   \cdots 0    & \Gamma _{m,1} \cdots \Gamma _{m,k_m}
                             \end{array}     \right) ;\]
and $d$ is a diagonal $(d_1,...,d_m)$. For the existence of the
diagram, first note that we have such maps $(\Gamma ,
\mathrm{id})$ and $d$ at the characteristic level due to: the
simplicity of $\overline{N}\leftarrow \overline{N}^{W/S}$; the
simplicity of $f$ at distinguished nodes; and the minimality of the
log prestable curve at nondistinguished nodes. Now to lift such
maps at chart levels, start with a chart in the middle and then
build the rest of the charts.

The minimality at distinguished nodes means that
 there is no nontrivial common divisors of positive integers $\Gamma _{i,k_1},...,\Gamma _{i,k_i}$ for all $i=1,...,m$.
 Let $\{ p_{i,j}\}$ be the set of distinguished nodes on $C_{\bar{s}}$, which are mapped into the $i$-th component
 of $W^{\mathrm{sing}}_{\bar{s}}$.
Since $d_i = l_{i,j}\Gamma _{i,j}$, $d_i$ must be the least common
multiple of $l_{ij}, \forall j$, where
$l_{i,j}=\overline{f^b}_{p_{i,j}}$ comes from the positive
multiplication map
$$\overline{f^b}_{p_{i,j}}: \ZZ\cong \overline{f^*(M_W/\pi ^*N)}_{p_{ij}} \ra
\ZZ \cong \overline{M_C/\pi ^* N}_{p _{ij}}$$
induced by $f^b$.

 \subsubsection{}
   Let a FM type space $W/S$ be equipped with a map $\pi _X: W\ra X$ from $W$ to a scheme $X$.
A log prestable map \eqref{twostars}  is
      called a  {\em log stable map}
       if the stability condition holds, i.e.,
       for each $s\in S$, the group of automorphisms $(h,\varphi )$ is finite
       where:

       \begin{itemize}

       \item  $h$ is an automorphism of $((C,M_C)/(S,N))_{\bar{s}}$
       preserving $n$-labeled points ${\bf p}_{\bar{s}}$.

       \item  $\varphi$ is an automorphism of $((W,M_W)/(S,N))_{\bar{s}})$
       preserving $W_{\bar{s}}\ra X$.

       \item $\varphi \circ f_{\bar{s}} = f_{\bar{s}} \circ h$.

       \end{itemize}

\medskip

We will see in \ref{auto} that the above stability holds if and only if
the underlying automorphism stability holds. That is, $\underline{f}$ is stable
if and only $f$ is stable.

\subsubsection{}\label{LogStableMapstoX} Let $\beta \in A_1(X)/\sim ^{\mathrm{alg}}$.
   A log prestable map \eqref{twostars}
      is called a $(g,n,\beta )$ {\em log stable map to a FM space $W/S$ of a smooth projective $\bf k$ variety $X$}
       if:

\begin{itemize}

    \item  $W/S$ is a FM space of $X$ over $S$.

    \item  Stability Condition holds.

    \item  $(\pi _X \circ f)_*[C_{\bar{s}}]=\beta$ for every $s\in S$.

\end{itemize}

  \medskip



\subsubsection{}\label{LogStableRamifiedMaps} Fix $\mu =(\mu_1,...,\mu _n)$.
A log stable map \eqref{twostars}  to a FM space $W/S$ of $X$
      is called a {\em log stable $\mu$-ramified map}
       if  $\underline{f}$
is stable $\mu$-ramified map over $S$ as in \ref{RamifiedMaps}.

\medskip

When $\mu = (1,...,1)$, we call it a {\em log stable unramified
map}. Note that in this case, indeed,  $f$ is {\em log unramfied}, that is,
   $f^*\Omega _{W/S}^\dagger \ra \Omega _{C/S}^\dagger $ is surjective at every $p\in C$.


 \subsection{Remarks}\label{remarks}

\subsubsection{} Every usual stable map to a smooth projective $\bf k$ variety $X$ from a prestable curve can
be an underlying stable map of a unique log stable map.


\subsubsection{}\label{loc.closed}

Suppose that  a log prestable curve $(C,M_C)/(S,N)$; an extended twisted FM
space $(W,M_W)/(S,N)$; and an admissible map $f :C/S\ra
W/S$ are given. Then there is a unique locally closed subscheme $S'$ of $S$ where
$f$ is a log prestable map over $(S',N_{|_{S'}})$ satisfying the universal
property: If $Z\ra S$ is a scheme morphism such that $f_{|_{Z}} :
(C_{|_Z}, (M_C)_{|_Z}) \ra (W_{|_Z}, (M_W)_{|_Z})$ is a log
prestable map over $(Z, N_{|_Z})$ if and only if $Z\ra S$ uniquely
factors as $Z\ra S'\ra S$. Here, for example, $N_{|_Z}$ is the log structure of the
pullback of $N$ under $Z\ra S$.

First of all the above statement makes sense because of \ref{description} where we have seen that
$f^b$ is determined by $\underline{f}$ and log structures on $C/S$ and $W/S$.
The minimality and the corank = \#
nondistinguished node condition are open conditions. Let $V$ be
the open subscheme of $S$ where these conditions are satisfied.
Then we need to consider only the condition that $f$ becomes a log
morphism. Since $f^b(\log x)=l\log u$ and $f^b(\log y)=l\log v$,
the condition becomes the equality $\log x + \log y = l \log u +
l\log v$ in $N$, which defines a locally closed subscheme $Z$ in
$V$ by the following reason.

\medskip

Let $M$ be a fine log structure on a scheme $Y_2$ and $\varphi$ be
a scheme morphism from a scheme $Y_1$ to $Y_2$. Then
$\overline{\varphi ^*M }= \varphi ^{-1}(\overline{M})$. This together with
$M_{\bar{y}}\cong \overline{M}_{\bar{y}}\oplus \mathcal{O}_{\bar{y}}^\times$, $y\in Y_2$,
implies
that the requirement of the \lq equality' of two sections of $M$
defines a  closed subscheme of an open subscheme in $Y_2$.


\section{Stacks}

In this section, we show that the moduli stacks of log stable maps are
 log algebraic (resp. proper DM) stacks over $\Lambda$
 whenever the corresponding stacks of underlying stable maps
are algebraic (resp. proper DM) over $\Lambda$. Since some stacks
considered here are not separated, we use terminology algebraic
stacks.

\subsection{Log stacks}   Following \cite{F.Kato_LogSmoothCurve},
we introduce log stacks and examples.

\subsubsection{} Let  $(\mathrm{Sch}/\Lambda )$ be the category of
locally noetherian schemes over $\Lambda$. A pair $(\mathcal{F}, L
)$ is called a {\em log stack} if: $\cF$ is a stack over $(\Sch
/\Lambda)$, and $L$ is a functor from $\mathcal{F}$ to the category
$LOG$ of fine log schemes over $\Lambda$, whose morphisms are strict log morphisms, making the diagram
\[{ \xymatrix{ \mathcal{F} \ar[r] \ar[d]_L & (\Sch /\Lambda ) \\
                   LOG \ar[ur]_{\text{forgetful}} &
}}\] commute. Furthermore when
$\mathcal{F}$ is an algebraic stack, it is called a log algebraic
stack. Here a stack $\mathcal{F}$ over $(\mathrm{Sch}/\Lambda )$
is said to be {\em algebraic} if
   the diagonal $\mathcal{F}\ra \mathcal{F}\times _{(\Sch/\Lambda)}\mathcal{F}$
   is representable and of finite presentation (see \S 4 in \cite{DM} for the definition),
   and it allows a smooth cover by a scheme.

   \medskip

  A log scheme $(X,M)$  can be considered as a log stack $(h_X, L_M)$, where
$h_X (Y) = \mathrm{Hom}_{(\Sch /\Lambda ) }(Y,X)$ and $L_M(f:Y\ra
X) = (Y, f^*M)$.

\subsubsection{}   For a fine log scheme $Y$,  $\mathcal{L}_{Y}$ will
   denote the stack of fine log schemes over $Y$:
   The objects are fine log schemes over the log scheme $Y$, and morphisms
   from $Z/Y$ to $Z'/Y$ are strict log morphisms
   $h: Z \ra Z'$ over $Y$. The log stack $\mathcal{L}_{Y}$ was
   denoted by $\mathcal{L}og _Y$ in \cite{Olsson_LogGeom}. To easy notation, we
   use the symbol $\mathcal{L}_{Y}$.

\medskip

For a log stack $(\cF,L)$, $\mathcal{L}_\cF$ is defined as: An
object is a pair $(x \in \cF (X), (X,M) \leftarrow L(x))$, where
$\cF (X)$ is the collection of objects of $\cF$ over $X$, and
the arrow is a (possibly non-strict) log morphism. A morphism from
$(x \in \cF (X), (X,M) \leftarrow L(x))$ to  $(y \in \cF (Y),
(Y,N) \leftarrow L(y))$ is a pair $(x \ra y, (X, M) \ra (Y,N))$,
where the first arrow is a morphism in the stack $\cF$, and the
second arrow is a strict log morphism for which the diagram
\[ \xymatrix{        (X,M) \ar[r] & (Y,N) \\
                             L(x) \ar [r]_{L(x\ra y)}\ar[u] & L(y)\ar[u]
                             }\] commutes.

Note that for a log scheme $(X,M)$, $\mathcal{L}_{(h_X,L_M)} $ is equivalent to $\cL _{(X,M)}$ as log stacks.

\begin{Lemma}
If a log stack $\mathcal{F}$ is algebraic, then $\mathcal{L}_{\mathcal{F}}$ is algebraic.
\end{Lemma}
\begin{proof} The following argument is due to Olsson.
Take a smooth cover of the algebraic stack $\cF$ by a scheme $Z$. Then we obtain
the diagram
\[\xymatrix {
\cL _\cF  \ar[r] & \cF \\
\cL _\cF \times _{\cF } Z \ar[r] \ar[u]&  Z .\ar[u] }\] Note that
the fibered category $\cL _\cF \times _{\cF } Z $ is equivalent to
$\mathcal{L}_Z$, where $Z$ is endowed with the log structure
induced from the cover. Since $\mathcal{L}_Z$ is an algebraic
stack due to \cite{Olsson_LogGeom}, we can apply Lemma
\ref{AlgStack} to the diagram to conclude that $\cL _\cF$ is an
algebraic stack.
\end{proof}

\begin{Lemma}\label{AlgStack} {\em (\cite{AOV})} Let $\mathcal{F}$ be  a stack. Suppose that there are:
 a morphism from $\mathcal{F}$ to an algebraic stack $\mathcal{G}$
and a smooth surjective map from a scheme $U$ to $\mathcal{G}$, for which $\mathcal{F} \times _{\mathcal{G}}U$ is
algebraic. Then $\mathcal{F}$ is algebraic.
\end{Lemma}

\subsection{Stacks of log prestable curves and extended log twisting FM type spaces}

\subsubsection{} As usual, we can define the category $\mathfrak{M}_{g}^{\mathrm{log}} $
of genus $g$, log prestable curves
over $(\mathrm{Sch}/\Lambda )$. Then using the inverse image of
log structures, we see that the category is fibered in groupoids
over $(\mathrm{Sch}/\Lambda )$. Furthermore, using the gluing of
sheaves, it is a stack. In fact, the stack $\fM _g^{\mathrm{log}}$
is equivalent to $\mathcal{L}_{\fM _g}$. Hence it is a log
algebraic stack.

Let  $\mathfrak{M}_{g, n}^{\mathrm{log}}$ be the algebraic stack
$\mathfrak{M}_{g}^{\mathrm{log}}\times _{\mathfrak{M}_g}\mathfrak{M}_{g,n}$ so that
there is no interesting log structures on markings.

\subsubsection{}\label{cB}

We consider a stack $\cB$ of certain $(W/S, W\ra X)$, where $W/S$
are log FM type spaces, and $W\ra X$ are maps from $W$ to a fixed
$X$. A morphism between them is a cartesian diagram preserving
$X$.
 By definition in \ref{FMtypespaces}, any object $W/S$ in $\cB$ can be realized as
 an underlying space of an extended log twisted FM type
 space. Suppose that $\cB$ is an algebraic stack.

Since $\cB$ is a log stack by the canonical log structures,
we can consider $\cL _{\cB}$. Let $\cB
^{\mathrm{tw}}$ (resp. $\cB ^{\mathrm{etw}}$) be the full substack
of $\cL _{\cB} $ whose objects are   (resp. extended) log twisted
FM type spaces whose underlying spaces are in $\cB$. By Lemma
\ref{chart}, they are open substacks of the algebraic stack $\cL
_{\cB}$, and hence they are also algebraic stacks.

Assume that there is a smooth scheme $B$ over $\Lambda$ and a
smooth morphism $B\ra \cB$,  defined by a \lq universal'
family $U/B$. Then,     $\cB ^{\mathrm{tw}}$ and $\cB
^{\mathrm{etw}}$ are smooth over $\Lambda$. Indeed, we can
formulate a smooth versal space as following. Let $W/k(\bar{p})$
be the pullback of $U$ at a point $p\in B$. Then, a formal versal
space of $\cB ^{\mathrm{etw}}$ at $W/k(\bar{p})$ with an  extended
log twisting (\ref{extendedLogTwisting})
  is
 \[ \mathrm{Spf} \hat{\mathcal{O}}_{\bar{p}} [[ x_1,...,x_m, y_1,...,y_{m'}]] / (x_1^{d_1} - \tau _1,...,
x_m^{d_m} - \tau _m), \] where $x_i, y_j$ are indeterminates;
\[ (d_1,...,d_m,\underbrace{0,...,0}_{m'}): \overline{N}_{\bar{p}}^{U/B} \ra \overline{N}_{\bar{p}};\]
$\hat{\mathcal{O}}_{\bar{p}}\ni \tau _i$ are $\alpha ^{U/B}(e_i)$
for a chart $\NN ^m  \ra N^{U/B}$ at $\bar{p}$; and $m=$ the number of
$\mathrm{Irr}U^{\mathrm{sing}}_{\bar{p}}$.

\subsubsection{}

Note that $\mathfrak{M}_{g,n}^{\mathrm{log}}\times _{LOG} \cB
^{\mathrm{etw}}$ is equivalent to the stack whose objects are
pairs of log prestable curves $(C,M_C)/(S,N)$ and extended log
twisting FM type spaces $(W, M_W)/(S, N)$. Since
$\mathfrak{M}_{g,n}^{\mathrm{log}}\times _{LOG} \cB
^{\mathrm{etw}}$
 is a fiber product of algebraic stacks over an algebraic stack, it is algebraic.
It is formally smooth over $\Lambda$, because the projection
$\mathfrak{M}_{g,n}^{\mathrm{log}}\times _{LOG} \cB
^{\mathrm{etw}} \ra \cB ^{\mathrm{etw}}$ is smooth by Proposition
3.14 in  \cite{K.Kato}, and $\cB ^{\mathrm{etw}}$ is formally
smooth over $\Lambda$.

\subsubsection{}\label{auto}

Consider a log twisting (\ref{LogTwisting}); assume that the
charts are global charts over $S$. Then the automorphism functor
$$\mathrm{Aut}_{W/S}( W/S, N\stackrel{d}{\leftarrow} N^{W/S})$$
over $W/S$, i.e., fixing $W/S$, is representable by \[
\mathbf{Spec}\mathcal{O}_S [ z_1^{\pm 1},...,z_m^{\pm 1}] / (z_i
^{d_i} - 1, z_i \alpha (e_i) - \alpha (e_i))\] over $S$, where
$\alpha :N\ra \co _S$ is the structure map of the log structure.
Similarly,
$$\mathrm{Aut}_{(C/S,W/S)} (C/S, W/S, N^{C/S}\stackrel{(\Gamma ,
\id )}{\ra}N\stackrel{(d,0)}{\leftarrow} N^{W/S}),$$ fixing all
underlying scheme structures, is representable by \[
\mathbf{Spec}\mathcal{O}_S [ z_1^{\pm 1},...,z_m^{\pm 1}] / (z_i
^{G_i} - 1, z_i \alpha (e_i) - \alpha (e_i)) \] over $S$, where
log twistings are given as in (\ref{DomainTarget}) with the global
charts over $S$, and $G_i=\mathrm{GCD}(\Gamma _{i,j}, \forall j)$.

\subsection{The stack of log stable maps}\label{StackLogStableMaps}


Denote by $\mathcal{U}$ be the stack of all pairs $(W/S, \sigma )$, where $W/S$ is an object in $\mathcal{B}$
and $\sigma$ is a section of $W\ra S$. The arrows are morphisms in $\mathcal{B}$ preserving sections.
Define the category
$\overline{M}_{g,n}^{\mathrm{log}}({\mathcal{U/B}})$
of
$(g,n)$ log stable maps to FM type spaces in the stack $\cB$. A
morphism from $((f: (C',M_{C'}, {\bf p}' )\ra
(W',M_{W'}))/(S',N_{S'}))$ to $((f:  (C,M_C, {\bf p}) \ra
(W,M_W))/(S,N_S)))$ is a commutative diagram
\[\xymatrix{
     (C', M_{C'}, {\bf p}' )  \ar[rr]_(.6){f'}\ar[dr]  \ar[dd] &            &  (W',M_{W'})   \ar[dd] \ar[ld] \ar[r]& X \ar[dd]^{=} \\
                                       & (S',N_{S'})  \ar[dd] &                       &  \\
    (C,M_C, {\bf p})  \ar'[r][rr]_(.3){f} \ar[dr]  &           &  (W,M_W)   \ar[ld]  \ar[r]    & X   \\
                                       &      (S,N_S)    &                       &
}\] where two side squares with the edge $(S', N_{S'}) \ra (S,
N_S)$ are fiber products in log sense preserving labeled
points, and the log structure $N_{S'}$ of $S'$ is naturally isomorphic to the pullback of
the log structure $N_S$ of $S$. Then it is a log stack over
$(\mathrm{Sch}/\Lambda )$; there is a natural
map
\[     \overline{M}_{g,n}^{\mathrm{log}}({\mathcal{U/B}}) \ra \mathfrak{M}_{g,n}^{\mathrm{log}}\times _{LOG} \cB ^{\mathrm{etw}}. \]

Similarly define the stack $\overline{M}_{g,n}(\mathcal{U/B})$ by
taking off all log structures. There is a natural map
\[    \overline{M}_{g,n}({\mathcal{U/B}}) \ra \mathfrak{M}_{g,n}\times _{(\Sch /\Lambda)} \cB .\]

Note that since $f^b$ is uniquely determined by $\underline{f}$
and log structures on its source and target, as seen in
\ref{description},
$\overline{M}_{g,n}^{\mathrm{log}}({\mathcal{U/B}})$ is a full
substack of
$$\overline{M}_{g,n}(\mathcal{U/B})
\times _{(\fM _{g,n}\times \cB)} (\fM _{g,n}^{\mathrm{log}} \times
_{LOG}\cB ^{\mathrm{etw}}).$$

In what follows, a DM stack over $\Lambda$ means an algebraic stack
over $\Lambda$ for which
   the diagonal $\mathcal{F}\ra \mathcal{F}\times _{(\Sch/\Lambda)}\mathcal{F}$
   is separated and unramified.

\begin{Thm}\label{Thm1}

If $\overline{M}_{g,n}(\mathcal{U/B})$ is an algebraic stack
(resp. a proper DM stack) over $\Lambda$, then so is
$\overline{M}_{g,n}^{\mathrm{log}}(\mathcal{U/B})$.

\end{Thm}

\begin{proof} Since $\overline{M}_{g,n}(\mathcal{U/B})$ is a full substack of
an algebraic stack, the isomorphism functors are representable and of finite presentation.
Now, the algebraic stack part of the statement follows from Remark
\ref{loc.closed}.
        The properness follows from that of  $\overline{M}_{g,n}(\mathcal{U/B})$
         and Lemma \ref{proper} since $\overline{M}_{g,n}^{\mathrm{log}}(\mathcal{U/B})$ is
         of finite presentation over $\Lambda$.  Due to the finiteness of automorphisms,  the diagonal is
         unramified. \end{proof}

\begin{Lemma}\label{proper} Let $R$ be a DVR
with an algebraically closed residue field; let $(W/S, W\ra X)$ be
a log FM type space, where $S=\Spec R$; and let $C/S$ be a
prestable curve over $S$. Suppose that a log stable map $f_\xi :
(C_\xi,M_{C_\xi }) \ra (W_\xi ,M_{W_\xi })$ over $(\xi, N_\xi)$ is
given, where $\xi$ is the generic point of $S$. Assume that there
is a stable admissible map $\underline{f}:C\ra W$ over $S$,
extending
 the underlying map $\underline{f_\xi }$.
Then  there exists a unique pair of a minimal log prestable curve $(C, M_C)/(S,N)$ and
an extended log twisted FM type space $(W, M_W)/(S,N)$, extending $(C_\xi, M_{C_{\xi}})/(\xi , N_{\xi})$ and
$(W_\xi, M_{W_{\xi}})/(\xi , N_{\xi})$, respectively,
such that  the stable admissible map  $\underline{f}$ becomes a log morphism over
$(S,N)$.
\end{Lemma}

\[\xymatrix{ (C_\xi, M_{C_\xi}) \ar[rr]_{f_{\xi}} \ar[dr] & & (W_\xi , M_{W_\xi}) \ar[dl]     & C
                     \ar[rr]_{\underline{f}}\ar[dr]  & & W \ar[dl]\\
                     & (\xi, N_{\xi})    &    &  & S &
                     }\]

\begin{proof}
Let $p$ be the closed point of $S$. We use the following index sets:

\[ \begin{array}{ccl} I_1 \text{ (resp. $I_1\coprod I_2$) } & = & \text{ an index set for the components } \\ & & \text{
of the singular locus of $W$ at $\bar{\xi}$ (resp. $p$). } \\
 I_1' \text{ (resp. $I_1'\coprod I_2'$) } &=&  \text{ an index set for the nondistinguished nodes} \\ & & \text{ of $C$ at $\bar{\xi}$ (resp. $p$).} \\
 I_1'' \text{ (resp. $I_1''\coprod I_2''$) } &=&  \text{ an index set for the distinguished nodes} \\ & & \text{ of $C$ at $\bar{\xi}$ (resp. $p$). }
 \end{array} \]

 Then by \ref{description} we may assume that the following compatible charts are given:
 \[ \begin{CD} N_{\bar{\xi}}^{C/S} @>>> N_{\bar{\xi}} @<<< N_{\bar{\xi}} ^{W/S} \\
      @AAA    @AAA @AAA \\
      \NN ^{I_1'' + I_1'} @>>{(\Gamma, \mathrm{id})}> \NN ^{I_1+I_2'} @<<{(d^{(1)},0)}< \NN ^{I_1} , \end{CD}\]
      where $d^{(1)}$ is a monoid homomorphism from $\NN ^{I_1''}$ to $\NN
      ^{I_1}$, and the sums of index sets is used for the disjoint
      unions of them, for simplicity.

For $i \in I_2''$, let $(u_i, v_i)$ be a special coordinate pair
at the node $p_i$ corresponding to $i$, and let $(x_i,y_i)$ be the
part of coordinates at $f(p_i)$ as in the admissible condition
so that $x_i=u_i^{l_i}$ and $y_i=v_i^{l_i}$ under $f$. For $j\in
I_2$, define $d^{(2)}_j = \mathrm{LCM}(l_i , \ \forall\,  i \mapsto j) $,
where $i\mapsto j$ means that $f(p_i)$ is in the component of
$W_{p}^{\mathrm{sing}}$ corresponding to $j\in I_2$.

  First define
a prelog structure \[ \alpha : P:= \NN ^{I_1+I_1'+I_2 + I'_2} \ra \mathcal{O}_S, \]
 where $\alpha$ is determined by \[ \alpha (e_j) = \begin{cases}   0 & j \in I_1 + I_1' \\
 (\alpha ^{W/S}  (e_j))^{1/d_j^{(2)}} & j\in I_2 \\
  \alpha ^{C/S}(e_j) & j\in I_2'
   \end{cases}
  \] where $\alpha ^{W/S}$ (resp. $\alpha ^{C/S}$) is
  the structure map from $N^{W/S}$ (resp. $N^{C/S}$) to $\mathcal{O}_S$,
  and  $(\alpha ^{W/S}  (e_j))^{1/d_j^{(2)}}$
  is a root (we choose a choice of the roots) whose $d_j^{(2)}$-th power is $\alpha ^{W/S}(e_j)$.
  Set $N=P^a$, the log structure associated to $P$.
  Now we construct a log prestable curve structure on $C/S$
  and an extended log twisting on $W/S$.
  It amounts to establishing
  certain log morphisms $N^{C/S}\ra N \leftarrow N^{W/S}$.
  We define them by homomorphisms between their charts as in
  diagram
  \[ \xymatrix { N^{C/S} \ar[rr] & & N & & \ar[ll] N^{W/S} \\
     \NN ^{I_1''+I_2''+I_1'+I_2'} \ar[u]\ar[rr]_{(\Gamma _1,\Gamma _2, \id, \id )}
   &  & \NN ^{I_1+I_2+I_1'+I_2'} \ar[u] & & \NN ^{I_1+I_2} \ar[u] \ar[ll]^{\ \  \ \ (d^{(1)}, d^{(2)},0,0)}
                   }\]
 where $\Gamma _2 (e_i) = \frac{d_j^{(2)}}{l_i} e_j$ for $i\in I_2''$, $i\mapsto j$,
 and $d^{(2)}=(d_j ^{(2)})_{j \in I_2}$.
The existence part of Lemma is verified.

 When $N$ is defined, the ambiguity occurs only
 in the choices of roots $\alpha ^{W/S}(e_j)^{1/d_j^{(2)}}$. Let $N'$ be a constructed one, using
 another choice. Then there is a unique isomorphism $N\ra N'$ commuting with maps from
 $N^{C/S}$ and $N^{W/S}$: The isomorphism is determined by
 \[\begin{array}{ccc} N & \ra & N' \\
           e_j & \mapsto & e_j + \alpha (e_j)/\alpha ' (e_j),  \end{array} \]
 where $\alpha '$ is the log structure map of $N'$, and $j\in I_2$ so that
 $ \alpha (e_j)/\alpha ' (e_j)\in \mathcal{O}_S^\times   $.
\end{proof}

\subsubsection{Remarks}\label{DMstack}  1. Assume that
 the stability condition of admissible maps to $U/B$, with respect to $\pi _X$
 (i.e., fixing $X$), is an open condition. Then we claim that
$\overline{M}_{g,n}(\mathcal{U/B})$ is a DM stack over $\Lambda$
by the following argument. Consider the stack
$\overline{M}_{g,n}(U/B; X)$ of admissible maps to the rigid
target $U/B$, {\em stable with respect to $\pi _X$}, which is a
full substack of the Kontsevich moduli space
$\overline{M}_{g,n}(U)$ of $(g,n)$  stable maps to $U$. Then the
stack is a DM stack over $\Lambda$ by Theorem 2.11 in \cite{Li1}.
This in turn shows that $\overline{M}_{g,n}(\mathcal{U/B})$ is a
DM  stack over $\Lambda$. Here the quasi-separatedness can be
directly shown.

\medskip

 2. The explicit description in \ref{description}
  shows that when $S$ is a geometric point $\Spec \mathbf{k}$,
  then for an underlying admissible map $\underline{f}$,
  the number of isomorphism classes of log prestable maps realizing $\underline{f}$
  is finite. The numerical data $l, d, \Gamma$ are uniquely determined. What remain
  indetermined are maps from $N^{W/S}$ and $N^{C/S}$ to $N$. The choices, however,
  are finite since $f$ is a log morphism. We illustrate the reason by an example, for simplicity.
  Suppose that the ranks of $N^{W/S}$ and $N$ are 1 so that we may write a map
  from $N^{W/S} = \NN \oplus {\bf k}^\times $ to $N=\NN \oplus {\bf k}^\times$
  by $e_1\mapsto de_1$. We express a map $N^{C/S}=\NN ^m\oplus {\bf k}^\times$
  to $N$ by $e_i \mapsto \Gamma _i e_i + \rho _i$. Note that $\rho _i$ must satisfy constraint
  $\rho _i ^{l_i} =1$, since $f$ is a log morphism. This shows the finiteness.
  Some of them could be isomorphic under maps $N\ra N$ sending $e_1\mapsto e_1 + \rho$,
  where $\rho\in {\bf k}^\times$ must satisfy the condition $\rho ^d =1$. Hence, in this example, there are
   $|(\Pi _i \ZZ / l_i \ZZ )/(\ZZ /d\ZZ)|$ many realizations of log prestable maps up to isomorphisms,
   where the action of $\ZZ /d\ZZ$ on $\ZZ / l_i \ZZ$ is given
  by $a\cdot (a_1,...,a_m) = (\Gamma _1a + a_1,...,\Gamma _m a + a_m)$,  $a\in \ZZ /d\ZZ$
  and $a_i\in \ZZ /l_i\ZZ$.

\medskip

3. Using the deformation and obstruction theory at the long exact
sequence in section \ref{ob}, one may  try directly to prove the
first part of Theorem \ref{Thm1}, using Artin's theorem in
\cite{Artin}. We do not pursue this approach here.

\section{Perfect Obstruction Theory}\label{ob}
\subsection{}\label{ob1}
In this subsection, we show that there is a natural perfect
obstruction theory on the log algebraic stack
$\overline{M}_{g,n}^{\mathrm{log}}(\mathcal{U}/\cB )$. The
 method parallel to \cite{BF, B} will work if the cotangent
complexes are replaced by the logarithmic cotangent complexes, as
following. We first consider the diagram
\[ \begin{CD} \mathcal{C} @>>f> \mathcal{U}^{\mathrm{etw}} \\
                    @VV{\pi}V @. \\
                    \cM  @.
   \end{CD}\]
where $\mathcal{C}$ (resp. $\mathcal{U}^{\mathrm{etw}}$) is the
universal family of the moduli log stack $\cM :=
\overline{M}_{g,n}^{\mathrm{log}}(\mathcal{U}/\cB )$ (resp. $\cB
^{\mathrm{etw}}$) and $f$ is the evaluation map. We regard $f$ and
$\pi$ as log morphisms between log stacks. Let $\mathfrak{C}$ be
the universal family of $ \fM _{g,n}^{\mathrm{log}}$ and let
$\fMB$ be $\fM _{g,n}^{\mathrm{log}}\times _{LOG}\cB
^{\mathrm{etw}}$.
 There is the composite of natural maps
\[ f^*L ^{\bullet}_{\mathcal{U} ^{\mathrm{etw}}/\mathcal{B} ^{\mathrm{etw}}}
\ra L^{\bullet} _{\mathcal{C}/\cB ^{\mathrm{etw}}}\ra
L^{\bullet}_{\mathcal{C}/\mathfrak{C} \times _{LOG}\cB
^{\mathrm{etw}} } \cong \pi ^* L ^\bullet _{ \cM /\fMB }
\]
between logarithmic cotangent complexes. See
\cite{Olsson_Cotangent} for the definition of logarithmic
complexes and functorial properties. Taking the tensor product of
the composite and the relative dualizing sheaf $\omega _{\pi}$ of
$\pi$, we obtain an element in
\[\mathrm{Hom}_{D^b(\mathcal{C})}(f^*L^{\bullet}_{\mathcal{U} ^{\mathrm{etw}}/\mathcal{B} ^{\mathrm{etw}}} \ot \omega _{\pi} [1],
\pi ^* L ^\bullet _{\cM /\fMB }  \ot \omega _{\pi} [1]).\] This in
turn yields an element in
\[\mathrm{Hom}_{D^b(\mathcal{M})} (R\pi _*f^*L^{\bullet}_{ \mathcal{U} ^{\mathrm{etw}}/\mathcal{B} ^{\mathrm{etw}} } \ot \omega _{\pi}[1],
L ^\bullet _{ \cM /\fMB} )\] since
\[ \pi ^* L ^\bullet _{ \cM /\fMB }  \ot \omega _{\pi}[1]\cong \pi ^! L
^\bullet _{ \cM/\fMB } \] and $R\pi _!$ is left adjoint to $\pi
^!$.  Finally by Grothendieck-Verdier duality we have a natural
homomorphism
\begin{equation}\label{perfect} E^\bullet \ra L^{\bullet\ge -1}_{\cM /\fMB}
\end{equation}
where $E^\bullet := (R\pi _* f^* T_{\mathcal{U}
^{\mathrm{etw}}/\mathcal{B} ^{\mathrm{etw}}}^\dagger)^\vee$ and
$L^{\bullet \ge -1}_{\cM /\fMB}$ is two-term $[-1,0]$ truncation of the
logarithmic relative cotangent complex. Note that the latter
complex is isomorphic to the usual relative cotangent complex
since the map $\cM \ra \fMB$ is strict.
 The homomorphism (\ref{perfect}) of the
complexes is a perfect obstruction theory since the relative
deformation/obstruction theory for log morphisms is as expected as
in Theorem 5.9 of \cite{Olsson_Cotangent}, and $E^\bullet$ can be
realized as a two-term complex of locally free coherent sheaves.
This defines a virtual
fundamental class of
$\overline{M}_{g,n}^{\mathrm{log}}(\mathcal{U}/\cB )$ by \cite{LT, BF, B}.

\medskip

The absolute obstruction theory
$F^\bullet\ra L^{\bullet}_{\cM}$ can be obtained as in \cite{GP,
KKP} so that there is a distinguished triangle:
\begin{equation}\label{AbsoluteTriangle} \tau ^*L^{\bullet}_{\fMB} \ra F^\bullet \ra E^\bullet ,\end{equation}
where $\tau$ is the natural map $\cM \ra \fMB$.
Consider a deformation situation of the extensions over $\Spec
A[I]$ of a given $(f,C^\dagger ,W^\dagger, S^\dagger )$ over
$S:=\Spec A$, where $A$ is a reduced noetherian $\Lambda$-algebra;
$I$ is a finite $A$-module; and $A[I]$ is the trivial ring
extension of $A$ by $I$. Let $\mathrm{Aut}(C^\dagger \times
_{S^{\dagger}}W ^\dagger )$ be the set of automorphisms of the
trivial extension of  $C^\dagger \times _{S^{\dagger}}W ^\dagger$
over $\Spec A[I] ^\dagger$, whose restriction to $S^\dagger$ is
the identity. Here the log structure on $\Spec A[I]^\dagger$ is
given by the pullback of $\Spec A[I]\ra \Spec A$. Also, let
$\mathrm{Def}_I(C ^\dagger \times _{S^{\dagger}} W ^\dagger )$ be
the set of isomorphism classes of the extensions over $S^\dagger$.
Then combining Proposition 3.14 in \cite{K.Kato} and Theorem 5.9
in \cite{Olsson_Cotangent}, we obtain an $A$-module exact sequence
\begin{eqnarray*} &0& \ra \mathrm{Aut}_{I}(C^\dagger \times _{S^{\dagger}}W ^\dagger ) \ra
\mathrm{RelDef}(f)=R^0\pi _*f^*T_{W^\dagger/S^\dagger} \ot _{\mathcal{O}_S}I \ra \mathrm{Def}(f) \\
&\ra& \mathrm{Def}_I(C ^\dagger \times _{S^{\dagger}} W ^\dagger )
\stackrel{\varphi}{\ra} \mathrm{RelOb}(f) = R^1\pi _*f^*T_{W^\dagger/S^\dagger} \ot _{\mathcal{O}_S} I \ra
\mathrm{Ob}(f) \ra 0
\end{eqnarray*} where
$\mathrm{Ob}(f)$ is defined to be the cokernel of $\varphi$. When
$I\cong A$, we may also use (\ref{AbsoluteTriangle}) and
Proposition III, 12.2 in \cite{Hart} to derive the above long
exact sequence.

\subsection{Log admissible covers}\label{ob2}
 Let $A$ be an Artinian local ring over $\mathbf{k}$, with the
maximal ideal $\mathfrak{m}_A$. Consider a small extension $R$ of
$A$ by $I$ and an admissible map $f$, locally at a
distinguished node described by
\[\begin{array}{ccc}(A[x,y, z_1,...,z_{r-1}]/(xy-\tau )^{\mathrm{sh}}& \stackrel{f^*}{\ra} &
 (A[u,v]/(uv-t))^{\mathrm{sh}} \\ x, \ y &\mapsto & u^l,\ v^
l ,
 \end{array}\]
  where superscript
$\mathrm{sh}$ means the strict henselianization at the \lq
origin'.
 We want to compute a local obstruction of
extending $f$ to an admissible map over a given extended domain
and a given extended target near the node. The obstruction is $\tilde{\tau} - \tilde{t}^l \in
I$ as explained in \cite{GV} if the extensions are given
by $xy=\tilde{\tau}\in\mathfrak{m}_R$
and $uv=\tilde{t}\in \mathfrak{m}_R$ (locally at singular points), respectively.
Here $\tilde{\tau}$ and $\tilde{t}$ are extensions of $\tau$ and $t$, respectively.
 This vanishes if $f$
is a log morphism and the extended source and target are
logarithmic ones. Indeed, the equation $[\log \tilde{\tau} ] =
[l\log \tilde{t}]$ in $\overline{N}_R$ implies that
$\tilde{\tau}=(1+c)\tilde{t}^l$ where $c\in I$. Since
$\tilde{t}\in \mathfrak{m}_R$, we see that $\tilde{\tau} = \tilde{t}^l$.

\medskip
Provided that $f$ is a log morphism,
for a log extension on the target with $xy=\tilde{\tau}$,
there is a unique log extension with $uv=\tilde{t}$
such that $\tilde{\tau}=\tilde{t}^l$.
This is  due to the existence of element $\tilde{s}\in\mathfrak{m}_R$ such
that $\tilde{s}^d=\tilde{\tau}$ and $s^\Gamma = t$,
where $\tilde{s}$ is an extension of $s$.
Hence, an infinitesimal deformation of $(W,M_W)/(S,N)$ uniquely
determines  an infinitesimal deformation of $(C,M_C)/(S,N)$ at distinguished nodes.
In particular, this shows that when the target is a projective
smooth curve $X$, the natural map
\[\begin{array}{ccc} \LMmu &\stackrel{\Psi}{ \ra} & X[n] ^{\mathrm{tw}} \\
     (f: (C, M_C, {\bf p}) \ra (W, M_W))/(S,N) )&\mapsto & (W/S, N\leftarrow N^{W/S}, f({\bf p})) \end{array} \]
                       is \'etale, where $\LMmu$ is the
                       moduli stack of $(g,n,\beta)$ log stable
                       $\mu$-ramified maps
                       (\ref{LogStableRamifiedMaps}), and
                       $X[n]^{\mathrm{tw}}$ is the stack of
                       $n$-pointed log twisted stable FM spaces of $X$.
                       Here \lq stable' means
                       that there are at least two special points on screens of FM spaces. By the following Lemma
                       $X[n]^{\mathrm{tw}}$ is an open stack of $\mathcal{L}_{X[n]}$.
(Note that $X[n]^{\mathrm{tw}}$ is not separated over ${\bf k}$ if $n\ge 2$.)
Essentially, it is already proven in \cite{We} that the map $\Psi$ is \'etale.

\begin{Lemma} {\em (\cite{KKO})} Let $X$ be a smooth variety over
$\bf k$. The stack of FM spaces with $n$ distinct, smooth,
sections is equivalent to the scheme $X[n]$ with the universal
space $X[n]^+$.
\end{Lemma}

According to \ref{cB}, $X[n]^{\mathrm{tw}}$ is smooth over $\bf k$.
Therefore, the finite map from the smooth stack $\LMmu$ to $\overline{M}_{g,n,\mu }(\fX ^+/\fX ,\beta )$ is the normalization map.
The normalization is also constructed by the stack of balanced twisted covers in \cite{ACV}.

\section{Chain Type}

\subsection{}\label{chain1} In this section, we provide a modular desingularization of the main component of
the moduli space of elliptic stable maps to a projective space.
The main component is, by definition, the irreducible component of
the moduli space, containing all the elliptic stable maps with
smooth domains.  First, note that any prestable curve $C$ over
$\mathbf{k}$ of genus 1 has a unique subcurve $C_0$ which is
either arithmetic $g=1$ irreducible component or a loop of
rational curves. We call $C_0$ the essential part of the curve
$C$. Its dual graph looks like:
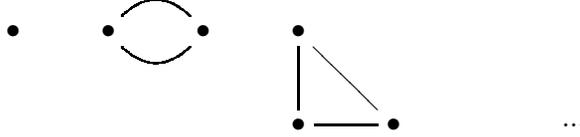
\[\xymatrix{\bullet &    \bullet \ar@{-}@/^1pc/[r] & \bullet \ar@{-}@/^1pc/[l] &       \bullet\ar@{-}[d]\ar@{-}[rd] &        &   &\\
     &           &                   &        \bullet\ar@{-}[r]               &\bullet  &  & ...
                            }\]
Note that the dualizing sheaf $\omega _{C_0}$ is trivial.
Let $\overline{M}_{1}^{\mathrm{log,\, ch}}(\fX ^+/\fX ,\beta)$ be
the moduli stack of $(g=1,n=0,\beta\ne 0)$ log stable maps $(f,C,W)$
satisfying the following conditions additional to those in
\ref{LogStableMapstoX}. For every $s\in S$:
\begin{itemize}
\item
Every end component of $W_{\bar{s}}$  contains  the entire
image of  the essential part of $C_{\bar{s}}$ under $f_{\bar{s}}$.

\item The image of the essential part of $C_{\bar{s}}$  is nonconstant.

\end{itemize}

Here, it is possible that some of irreducible components in the essential part are mapped to points.
Note that the dual graph of the target $W_s$ must be a chain.
Such a log stable map is called an elliptic log stable map to a chain type FM space $W$ of the
smooth projective variety $X$. Set
$\overline{M}_{1}^{\mathrm{log,\, ch}}(\fX ^+/\fX ):=\coprod _{\beta} \overline{M}_{1}^{\mathrm{log,\, ch}}(\fX ^+/\fX ,\beta)$.

\subsection{Proof of Main Theorem B}
By Theorem \ref{Thm1}, it is enough to prove the properties for
the moduli stack of the non-log version of elliptic log stable
maps to chain type FM spaces of $X$. It is a DM stack of finte
type over $\bf k$ by Remark 1 in \ref{DMstack}.

\medskip

{\em Properness.}  We use the
valuative criterion of properness. We show that the criterion
holds,  by adapting the properness argument in \cite{KKO}. Let $f_\xi :
C_\xi \ra W_\xi $ be an elliptic stable map to a chain type $W$
over a quotient field of the DVR $R={\bf k}[[t]]$, and let $\xi$ and $p$
be the generic point and the closed point of $R$, respectively. For
simplicity assume that $W_\xi = X\times \xi$.

- Uniqueness: Suppose that there are two extensions $f_i : C^{(i)}
\ra W^{(i)}$ of $f$ over $R$. To show that $f_1$ and $f_2$ are
equivalent over $f$, it is enough to verify that $W^{(1)}$ and
$W^{(2)}$ are isomorphic over $X$ and $R$. To prove it, we
construct sections of $W^{(i)}\ra \Spec R$. Consider two sections
passing through the end component of $W^{(i)}_p$. Also, consider,
for each ruled component, a section passing through the ruled
component. Let $N_i$ be two plus the number of ruled component of
$W^{(i)}_p$ unless $W^{(i)}_p=X$. Let $N_i=0$ if $W^{(i)}_p=X$. We
may assume that those $N_i$ sections are pairwise distinct at $p$.
Let $W$ be the FM type space of $X$ over $R$, separating each set
of $N_i$ sections: $W$ is defined to be $g^*X[N_1, N_2]^+$, where
$g: \Spec R\ra X[N_1, N_2]$ is the map associated to two sets of
sections. Here $X[N_1, N_2]$ denotes the compactification of
configurations of two sets of $N_i$ distinct labeled points, $i=1,
2$. The precise construction of $X[N_1, N_2]$ can found in
\cite{KKO}. The space $X[N_1, N_2]$ has the universal space
$X[N_1, N_2]^+$ with two sets of $N_i$ distinct sections. We claim
that $W$ is isomorphic to $W_i$ over $X$ and $R$. To see it, we
first remove two sections associated to the end component of
$W^{(2)}$. Then still, $W$ with the remaind sections is stable
since the domain of the stable map limit to the new target $W$ has
genus 1 and the stable map limit agrees with $f_i$ after the
contraction of suitable screens of the target $W$ and then stable
contraction of the domain curve. Now remove the rest of sections
associated ruled components of $W^{(2)}$. $W$ with $N_1$ sections
remains stable otherwise there will be a component of curve
$C^{(1)}_p$ which is mapped into a singular locus of $W^{(1)}$.
Thus, $W\cong W^{(1)}$ over $X$ and $R$. The same method shows
that $W\cong W^{(2)}$ over $X$ and $R$,
 and hence $W^{(1)} \cong W^{(2)}$ over $X$ and $R$.

- Existence: Consider each irreducible component $C_\gamma$ of the essential
subcurve of $C_\xi$ where $f_\xi (C_\gamma)$ is nonconstant.
The $f_\xi$ restricted to $C_\gamma$ is $\mu$-ramified, for some $\mu = (\mu _1,...,\mu _n)$.
Consider the screens created by taking the limit
of $f_\xi$ restricted to $C_\gamma$ as a $\mu$-ramified stable map. See \cite{KKO} for the construction.
There is a natural partial order on the set of all the screens for all $\gamma$.
Take the first one (which needs the smallest magnifying power to be visible).
It exits since $g=1$. Now consider two general sections passing through the first screen, and
then the stable map limit with the new target separating the sections. If necessary, take a further
expansion along the singular locus of the target at the special
fiber so that there are no domain components which are mapped into
singular locus of the new target. This process ends in a finite step
as in \cite{KKO} and yields a stable admissible map to a
chain type FM space of $X$ over $R$, which satisfies the desired requirements.

\medskip

{\em Smoothness.} This is simply because the relative obstruction
space in section \ref{ob} vanishes as following. Let $f:C^\dagger
\ra W^\dagger $ over ${\bf k}^\dagger$ be an elliptic log stable map to a chain type FM space $W$
of $X:=\PP ^r_{\bf k}$.  In what follows,
$T^\dagger$ denotes the log tangent sheaf
$T_{W^\dagger/{\bf k}^\dagger}$.

Decompose $C$ to be the union $\bigcup C_i$ of the essential part $C_0$ and
the irreducible rational components $C_i$, $i\ge 1$;  let $f_i$ be the map $f$ restricted
to $C_i$.
To prove $H^1(C, f^*T^\dagger)=0$, we claim that $H^1(C_i,
f^*T^\dagger)=0$ for all $i$ and the evaluation map \[ \bigoplus _i
H^0(C_i,f_i^*T^\dagger)\ra \bigoplus _{\{i,j: i\ne
j\}}f^*T^\dagger|_{C_i\cap C_j}
\] is surjective. For the latter, by  the induction argument
on the number of irreducible rational components $C_i$, $i\ge 1$,
it suffices that $H^1(C_i, f_i ^*T^\dagger\ot \mathcal{O}_{C_i}
(-p))=0$ if $p$ is a point in $C_i$ and $i\ge 1$. When $W=X$ case,
it  follows from the Euler sequence, since $H^1(C_0,
f_0^*\mathcal{O}_X(1) )=H^0(C_0, f_0^*\mathcal{O}_X (-1)\ot \omega
_{C_0})^\vee =0$ and $H^1(C_i, f_i^*\mathcal{O}_X(1) \ot
\mathcal{O}_{C_i}(-p))=0$ for all $i\ge 1$. Similarly, when $W$ is
singular, the claim follows  from the generalized  Euler
sequences:
\[\begin{array}{ccccccccc} 0   &\ra & \mathcal{O}_Y & \ra  & (\bigoplus\mathcal{O}_Y(1))\oplus \mathcal{O}_Y & \ra &
T^\dagger _{ |_{Y}} &\ra & 0 \\
0 &\ra& \mathcal{O}_Y(-E) &\ra& (\bigoplus\pi^*\mathcal{O}_{\PP ^r}(1))(-E) &\ra&
T^\dagger _{ |_Y} &\ra& 0 \\
0 & \ra& \mathcal{O}_Y(-E) &\ra& (\bigoplus\pi^*\mathcal{O}_{\PP
^r}(1)(-E))\oplus \mathcal{O}_Y(-E) &\ra& T^\dagger _{ |_Y} &\ra&
0
 \end{array}\]  for the logarithmic tangent sheaf restricted to the end, the root,
and a ruled component $Y$ of $W$, respectively. The explanation of the sequences is in order.
The first sequence is standard. The second one can be obtained from an isomorphism
\[ \pi ^*T_X (-E) \ra T^\dagger_{| _Y} .\] This isomorphism can be seen by a
local description of blowup $\pi : Y\ra X$ at a point, with the
exceptional divisor $E$ (for example, see the
proof of Lemma 1 in \cite{KLO}). The third one results from the gluing of the
previous two sequences.
\hfill$\Box$


\end{document}